\documentclass[11pt, reqno]{amsart}

\usepackage{graphicx}
\usepackage{epstopdf}
\usepackage{subcaption}
\usepackage[utf8]{inputenc}
\usepackage[T1]{fontenc}
\usepackage[english]{babel}
\usepackage{amsmath, amsthm, amssymb, amsfonts}
\usepackage{ae}
\usepackage{bbm}
\usepackage{caption}
\usepackage{array}
\usepackage[hmarginratio=1:1]{geometry}
\usepackage[hyphens]{url}
\usepackage[breaklinks=true]{hyperref}

\usepackage{mathrsfs}
\usepackage{placeins} 
\usepackage{float}
\usepackage[nodayofweek]{datetime}
\usepackage{enumitem}
\usepackage{mathtools}
\usepackage{dsfont}
\usepackage{comment}
\usepackage{csquotes}

\DeclareMathOperator{\Tr}{Tr}

\newtheorem{theorem}{Theorem}[section]

\newtheorem{lemma}[theorem]{Lemma}
\newtheorem{proposition}[theorem]{Proposition}
\newtheorem{corollary}[theorem]{Corollary}

\numberwithin{theorem}{section}
\theoremstyle{remark}

\usepackage[backend=bibtex, style=numeric-comp]{biblatex}
\appto{\bibsetup}{\sloppy}
\AtEveryBibitem{
  \clearfield{month}
}

\begin{document}

\allowdisplaybreaks

\title[Sums of Dirichlet eigenvalues]{Sums of Dirichlet eigenvalues in proximity of the disk and the ball}

\author{Matthias Baur}
\address{\textnormal{(M. Baur)} Institute of Analysis, Dynamics and Modeling, Department of Mathematics, University of
Stuttgart, Pfaffenwaldring 57, 70569 Stuttgart, Germany.}
\email{\href{mailto:matthias.baur@mathematik.uni-stuttgart.de}{matthias.baur@mathematik.uni-stuttgart.de}}

\begin{abstract}
We derive second order perturbation expansions of volume-normalized sums of the first $n$ Dirichlet eigenvalues for near-circular and near-spherical domains.  The disk (resp.~ball) is only a critical shape when $n$ is the last index of an eigenvalue cluster, and in dimension two and three our expansions yield a finitely checkable criterion deciding whether the disk or ball is a weak local minimizer or a saddle point. In two dimensions, this allows us to show that there is an infinite number of cases in which the disk is a saddle point for finite eigenvalue sums. 
\end{abstract}

\maketitle

% ------------------  Introduction  --------------------

%\allowdisplaybreaks

\section{Introduction and main results}

\subsection{Introduction}
We consider the Dirichlet Laplacian $-\Delta_\Omega^{\rm D}$ on bounded, open domains $\Omega\subset\mathbb{R}^d$, $d\geq 2$. Its eigenvalues shall be $\lambda_k(\Omega)$, $k\geq 1 $. As usual we count eigenvalues with multiplicities and assume them to be sorted in increasing order. We are interested in the shape optimization problem
\begin{align} \label{eq:inf_problem}
    \inf\bigg\{ \sum\limits_{k=1}^n \lambda_k(\Omega) \, : \, \Omega \subset\mathbb{R}^d \text{ open},\, |\Omega| = 1  \bigg\}.
\end{align}
Existence of a minimizer of \eqref{eq:inf_problem} is established by the works of Buttazzo-Dal Maso \cite{Buttazzo1993}, Bucur \cite{Bucur2012}, Mazzoleni-Pratelli \cite{Mazzoleni2013} (at least for the relaxed problem in the class of quasi-open sets).

In the case $n=1$, the well-known Faber-Krahn theorem \cite{Faber1923, Krahn1925, Krahn1926} states that the ball minimizes $\lambda_1(\Omega)$. Recently, it has been claimed \cite{Alexa2026} that the ball minimizes \eqref{eq:inf_problem} among convex domains for any $n\in\mathbb{N}$. Here, we refute this claim. The ball is not a minimizer of \eqref{eq:inf_problem} in an infinite number of cases and regardless of a restriction of the problem to convex sets. In fact, the ball is a critical shape of \eqref{eq:inf_problem} if and only if $n$ is the last index of an eigenvalue cluster of the ball\footnote{The following proof idea works for any dimension $d \geq 2$: Start with Hadamard's shape derivative formula for degenerate eigenvalues. For every eigenvalue cluster, there exists a deformation of the ball that leads to linear separation of the eigenvalues of that cluster. However, independent of the deformation, the trace of all eigenvalues of a cluster always experiences only a quadratic correction - the individual linear corrections within a cluster cancel out. A decomposition of the full eigenvalue sum into sums over eigenvalue clusters (with the last one being complete if and only if $n$ is the last index of that eigenvalue cluster) then proves the claim.}. However, to our knowledge, in the cases where the ball is a critical shape it is far from understood whether it is a minimizer of \eqref{eq:inf_problem} for $n\geq 2$. We explore this question perturbatively by considering near-circular and near-spherical domains in dimensions $d=2$ and $3$. Going up to second order corrections allows us to formulate a necessary condition for minimality of the ball and lets us identify cases where this condition fails and the ball is a saddle point.

Existing results on minimizers for eigenvalue sums either concern geometric normalizations different from the volume constraint, asymptotic optimality or certain partial sums. Laugesen and Siudeja \cite{Siudeja2011} proved that among triangles of given diameter, the equilateral triangle minimizes $\lambda_1(\Omega) + ... + \lambda_n(\Omega)$ and conjectured that the disk is the minimizer among all domains. In another paper \cite{Laugesen2011a}, they showed that the disk maximizes $(\lambda_1(\Omega) + ... + \lambda_n(\Omega)) |\Omega|^3 / I(\Omega)$ for any $n\in\mathbb{N}$ where $I(\Omega)$ denotes the moment of inertia about the center of mass. For \eqref{eq:inf_problem}, Freitas \cite{Freitas2017} proved asymptotics of the attained minimal eigenvalue sums, but it remains a conjecture that the ball is the asymptotic minimizer of \eqref{eq:inf_problem}, i.e.~minimizers of \eqref{eq:inf_problem} converge to the ball for $n\to\infty$. Larson \cite{Larson2019} proved asymptotic optimality of the ball if minimization in \eqref{eq:inf_problem} is restricted to convex sets. Numerical studies on minimizers of certain eigenvalue sums have been carried out \cite{Antunes2013,Osting2013,Osting2013a}, but these examine (weighted) sums of two/three consecutive Dirichlet eigenvalues rather than \eqref{eq:inf_problem}. Regarding conjectures about \eqref{eq:inf_problem}, we are only aware of Henrot's conjecture that $\lambda_2(\Omega)+\lambda_3(\Omega)$ and thus, by Faber-Krahn, \eqref{eq:inf_problem} for $n=3$ is minimized by the disk \cite{Henrot2006}.

In spirit, the perturbative analysis we are going to present here goes back to Lord Rayleigh himself who substantiated his claim that the disk minimizes $\lambda_1(\Omega)$ among domains of equal area\footnote{This would be proven later by Faber and Krahn and become the Faber-Krahn theorem.} by showing that small perturbations of the circular boundary of the disk increase the principal eigenvalue \cite[Vol. 1, Chapter IX, \S\S 209-210]{Rayleigh1894}. P{\'o}lya and Szeg{\H o}  \cite[Chapter VI]{Polya1951} applied this idea systematically to other isoperimetric problems. The perturbative analysis that we present for eigenvalues of near-circular and near-spherical domains is based on (and extended to dimension $d=3$ from) the planar exposition by Wolf and Keller \cite{Wolf1994} who used perturbations of a circular boundary written in a Fourier expansion to show that the disk is a local minimum for $\lambda_3(\Omega)$ under area constraint. Following their analysis, Berger \cite{Berger2014, Berger2015} showed that Dirichlet eigenvalues $\lambda_k(\Omega)$ can never be minimized by the disk or any disjoint unions of disks if $k\geq 5$. For Steklov eigenvalues, it has been recently shown through the same Fourier-type perturbations that the disk is never a maximizer for Steklov eigenvalues $\sigma_{k}(\Omega)$, $k\geq 2$, under area constraint \cite{Viator2018,Alland2026}. Bogosel, Henrot and Lucardesi \cite{Bogosel2018} use Fourier-type perturbations to the support function instead of the circular boundary and derive second-order corrections to show that the disk cannot be a local minimizer for Dirichlet eigenvalues $\lambda_k(\Omega)$ under diameter constraint except for a finite list of $k\in\mathbb{N}$. In all these results, the disk is excluded as a local minimizer by deriving first (second) order perturbation terms and checking whether the disk is non-critical or a saddle point. This idea is also the basis of this paper.

\subsection{Main results}

Let us denote $\mathbb{D} = \{ x\in \mathbb{R}^2 \, : \,\Vert x\Vert < 1\}$ and $\mathbb{B} = \{ x\in \mathbb{R}^3 \, : \,\Vert x\Vert < 1\}$ and let us introduce renormalized eigenvalue sums given by
\begin{align*}
    S_n(\Omega) := \sum\limits_{k=1}^n \lambda_k(\Omega)|\Omega|^{\frac{2}{d}} .
\end{align*}
Then \eqref{eq:inf_problem} is equivalent to
\begin{align*}
    \inf\{ S_n(\Omega) \, : \, \Omega \subset\mathbb{R}^d \text{ open} \} 
\end{align*}
by scaling of Dirichlet eigenvalues. We use this unconstrained form of the shape optimization problem in the following.

For a shape functional $F(\Omega)$, we adopt the notion from \cite{Bogosel2018} and call a shape $\Omega$ a \textit{weak local minimizer of $F(\Omega)$} if $\Omega$ is locally optimal for a certain admissible class of deformations. A necessary (sufficient) condition is that $\Omega$ is a critical shape and has non-negative (positive) second variation along all admissible directions. In this paper, we consider deformations of the boundary by Fourier modes or spherical harmonics and understand local optimality modulo translations (for details see Sections \ref{subsec:twodim_framework} and \ref{subsec:threedim_framework}). Furthermore, we call a shape $\Omega$ a \textit{saddle point of $F(\Omega)$} if $\Omega$ is a critical shape but there exists a deformation with negative second variation.

Our first result categorizes the disk in these terms for any renormalized eigenvalue sum $S_n(\Omega)$, $n\in\mathbb{N}$.

\begin{theorem}[Local optimality of the disk]\label{thm:local_minimizer_disc}
Let $n\in\mathbb{N}$.
\begin{enumerate}
\item If the $n$-th eigenvalue of the unit disk is the first instance of a double eigenvalue, i.e.~$\lambda_n(\mathbb{D})  = \lambda_{n+1}(\mathbb{D})$, then the disk is not critical and thus not a (weak) local minimizer of $S_n(\Omega)$.
\item If the $n$-th eigenvalue of the unit disk is a single eigenvalue or the second instance of a double eigenvalue, i.e.~$\lambda_{n}(\mathbb{D}) < \lambda_{n+1}(\mathbb{D}) $, then the disk is critical and it is a weak local minimizer of $S_n(\Omega)$ if 
$$d_j+d_{-j} > 0 \quad \text{ for all } j\geq 2$$ 
where $d_j$ is given by
\begin{align} \label{eq:local_minimizer_disc_case2_dj}
\begin{split}
d_j &=   \sum\limits_{j_{0p}^2 \leq \lambda_n(\mathbb{D})} 2\pi j_{0p}^2 \left(1 + j_{0p}\frac{J_j'(j_{0p})}{J_j(j_{0p})}   \right)  \\
& \qquad+ \sum\limits_{\substack{j_{mp}^2 \leq \lambda_n(\mathbb{D}), \\ m>0}} 4\pi j_{mp}^2 \left( 1 + \delta_{j= 2m} +  j_{mp}\frac{J_{m-j}'(j_{mp})}{J_{m-j}(j_{mp})} \delta_{j\neq 2m} \right).
\end{split}
\end{align}
If $d_{j}+d_{-j}$ is negative for some $j\geq 2$, then the disk is a saddle point of $S_n(\Omega)$.
\end{enumerate}
\end{theorem}

Here, the functions $J_m$ are Bessel functions of the first kind and $j_{mp}$ denotes the $p$-th positive zero of $J_m$.

Case (2) of the theorem follows from a second order perturbation expansion of $S_n(\Omega)$ near the disk (see Proposition \ref{prop:prop_Sn_exp}) and the condition stated is simply that all second order variations are positive\footnote{We exclude the Fourier mode $j=1$ which is always zero since it corresponds to infinitesimal translations.}.

The theorem allows one to numerically check which $n$ satisfy the condition $d_j+d_{-j} > 0$ for all $j\geq 2$ by evaluating $d_j+d_{-j}$ for a finite number of small $j$'s. This is sufficient as Lemma \ref{lem:dj_pos_2mmax} shows. Each $n$ for which the disk is critical then falls into one of two categories: either the disk is a weak local minimizer or a saddle point of $S_n(\Omega)$\footnote{This dichotomy assumes that no coefficient $d_j+d_{-j}$, $j\geq 2$, vanishes, otherwise second order information is inconclusive. Numerically, this does not occur for the values of $n$ we have checked.}. The first few $n$ give the following picture: \vspace*{2mm}

\begin{center}
\begin{tabular}{l|l}
$\mathbb{D}$ is a   & $n$ (checked up to 150) \\ \hline
weak local     & $1, 3, 6, 12, 17, 23, 30, 34, 42, 51, 61, 67, 78, 82, 90, 92, 94,$ \\
minimizer & $105, 107, 119, 127, 129, 142, 144, ... $  \\ \hline
saddle    & $5, 8, 10, 14, 15, 19, 21, 25, 27, 29, 32, 36, 38, 40, 44, 46, 48,$ \\
point     & $50, 53, 55, 57, 59, 63, 65, 69, 71, 73, 74, 76, 80, 84, 86, 88, 96,$ \\
     & $98, 100, 102, 104, 109, 111, 113, 115, 117, 121, 123, 125, 131,$ \\
     & $133, 135, 137, 139, 140, 146, 148, 150, ... $
\end{tabular}
\end{center}
\vspace*{2mm}

In the weak local minimizer category, we have confirmed the first seven terms by running the gradient descent shape optimization algorithm developed in \cite{Antunes2012}. In these cases, the disk appears not just to be a local minimizer but also the global minimizer of $S_n(\Omega)$. The second category, the saddle point category, seems to be much more densely populated than the weak local minimizer category. We can prove that it contains an infinite number of members.

\begin{theorem} \label{thm:saddle_point_infinitely_many}
    There are infinitely many $n$ for which the disk $\mathbb{D}$ is a saddle point of $S_n(\Omega)$. 
\end{theorem}

Specifically, we show this by restricting attention to elliptical ($j=2$) perturbations for $n$ where $\lambda_{n+1}(\mathbb{D})$ is simple and showing that $d_2+d_{-2}<0$ for any $n$ in this family (see Proposition \ref{prop:saddle_before_simple_n}).

To decide whether the category of weak local minimizers is infinite or finite is a much harder problem than for the saddle point category, since not just one perturbation coefficient $d_j+d_{-j}$ has to be controlled, but all of them. We are not able to decide this question in this work, but numerical computations (see Appendix \ref{app:numerical_addendum}) speak in favor of the conjecture that the number of weak local minimizer cases is also infinite. This is remarkable since all comparable results mentioned previously \cite{Berger2014,Berger2015,Bogosel2018,Viator2018,Alland2026} found that the number of weak local minimizer cases is finite in their respective problems.

In three dimensions, we prove the following result analogous to Theorem \ref{thm:local_minimizer_disc}.

\begin{theorem}[Local optimality of the ball]\label{thm:local_minimizer_ball3d}
Let $n\in\mathbb{N}$.
\begin{enumerate}
\item If $\lambda_n(\mathbb{B})  = \lambda_{n+1}(\mathbb{B})$, then the ball is not critical and thus not a (weak) local minimizer of $S_n(\Omega)$.
\item If $\lambda_{n}(\mathbb{B}) < \lambda_{n+1}(\mathbb{B})$, then the ball is critical and it is a weak local minimizer of $S_n(\Omega)$ if
$$D_{L} > 0 \quad \text{ for all } L\geq 2,$$ 
where $D_{L}$ is given by
\begin{align} \label{eq:local_minimizer_ball_case2_DL}
\begin{split}
D_{L} &= |\mathbb{B}|^{2/3} \sum\limits_{j_{l+1/2,p}^2 \leq \lambda_n(\mathbb{B})} (2l+1) j_{l+1/2,p}^2 \Bigg(\frac{1}{ \pi}  + \frac{(2l+1)}{4\pi (2L+1)}  \big(\langle l, 0, l,0|L,0\rangle \big)^2 \\
&\hspace{2.8cm}+ 2\sum\limits_{l'\neq l} j_{l+1/2,p}\frac{j_{l'}'(j_{l+1/2,p})}{j_{l'}(j_{l+1/2,p})} \frac{(2l'+1)}{4\pi (2L+1)}  \big(\langle l, 0, l',0|L,0\rangle \big)^2  \Bigg).
\end{split}
\end{align}
If $D_{L}$ is negative for some $L\geq 2$, then the ball is a saddle point of $S_n(\Omega)$.
\end{enumerate}
\end{theorem}

Here, the functions $j_l$ are spherical Bessel functions of the first kind and $\langle l, m, l',m'|L,M\rangle$ denote Clebsch-Gordan coefficients.

As in dimension two, case (2) of the theorem follows from a second order perturbation expansion of $S_n(\Omega)$ near the ball (see Proposition \ref{prop:prop3d_Sn_exp_quadratic}). Our expansion agrees in the case $n=1$ with the one derived by Goulian and Milner \cite{Goulian1996} for the principal eigenvalue $\lambda_1(\Omega)$, but our perturbation analysis for Dirichlet eigenvalues in the three dimensional case goes far beyond the principal eigenvalue.

Again, a finite check of coefficients $D_L$ is enough to categorize every $n$ (see Lemma \ref{lem:Dj_pos_2mmax}). The first few $n$ give: \vspace*{2mm}

\begin{center}
\begin{tabular}{l|l}
$\mathbb{B}$ is a    & $n$ (checked up to 350) \\ \hline
weak local     & $1, 4, 9, 10, 17, 29, 46, 66, 69, 93, 98, 99, 127, 134, 169, 219,$ \\
minimizer & $220, 243, 271, 278, 281, 306, 338,  ... $  \\ \hline
saddle    & $20, 34, 45, 53, 84, 116, 153, 156, 190, 199, 214, 254, 319, 347,$ \\
point     & $...$
\end{tabular}
\end{center}
\vspace*{2mm}

It is striking that the density of saddle point cases compared to weak local minimizer cases initially appears to be much lower in three dimensions than in two dimensions. However, numerics (see Appendix \ref{app:numerical_addendum}) suggest that saddle point cases dominate asymptotically over weak local minimizer cases. Coincidentally or not, the proof strategy that allows us to show that the saddle point category is infinite in two dimensions unfortunately breaks down in three dimensions. Ellipsoidal perturbations ($L=2$) of the ball asymptotically have positive second order variations and cannot be used to show that the ball is a saddle point of $S_n(\Omega)$ (see Proposition \ref{prop:D2toinfinity}). Understanding the sign behavior of $D_L$, $L\geq 3$, and thus the distribution of weak local minimizer and saddle point cases seems to require a rather deep understanding of the distribution of zeros of Bessel functions which is beyond the scope of this article. An analogue of Theorem \ref{thm:saddle_point_infinitely_many} for three dimensions thus remains an open problem. We conjecture that as in two dimensions both categories, weak local minimizers and saddle points, give infinite sequences.

\section{Second order perturbation expansions for the disk}

The goal of this section is to compute perturbation expansions up to second order for renormalized sums of eigenvalues $S_n(\Omega)$. Since $S_n(\Omega)$ is invariant under scaling of the domain, we focus on perturbation expansions of the unit disk.

\subsection{Two-dimensional perturbation framework} \label{subsec:twodim_framework}

We recall that the eigenvalues
$$0 < \lambda_1(\mathbb{D}) < \lambda_2(\mathbb{D}) \leq \lambda_3(\mathbb{D}) \leq ...$$
of the Dirichlet Laplacian on the unit disk are given by $\lambda_n(\mathbb{D}) = j_{mp}^2$ for some $m\in \mathbb{N}_0$, $p\in\mathbb{N}$, where $j_{mp}$ denotes the $p$-th positive zero of the $m$-th order Bessel function $J_m$. If $m=0$, the eigenvalues $\lambda_{n}(\mathbb{D}) = j_{0p}^2$ are simple and an associated eigenfunction is
\begin{align*}
u_{0p}(r,\varphi)=J_0(j_{0p}r).
\end{align*}
For $m>0$, every eigenvalue $\lambda_{n}(\mathbb{D}) = j_{mp}^2$ is a double eigenvalue. The corresponding eigenspace is then spanned by two eigenfunctions
\begin{align*}
u_{mp}^\pm(r,\varphi)= J_m(j_{mp}r) e^{\pm i m \varphi }. 
\end{align*}
We note that $J_m$ and $J_{m'}$ have no common positive zeros for $m\neq m'$ (Bourget's hypothesis, see \cite[\S 15.28]{Watson1995}, proved by Siegel \cite{Siegel1929}); the same fact guarantees that all denominators $J_j(j_{0p})$, $j\neq 0$, and $J_{m-j}(j_{mp})$, $j\neq 0,\, 2m$, appearing below and in Theorem \ref{thm:local_minimizer_disc} are non-zero.

In the following, we will consider near-circular domains that arise from perturbing the boundary of the unit disk. Let $\Omega$ be a star-shaped domain whose boundary is described by a curve 
$$\gamma: [0,2 \pi) \rightarrow \mathbb{R}^2, \quad \gamma(\varphi)=r(\varphi)\begin{pmatrix}
\cos(\varphi) \\ 
\sin(\varphi)
\end{pmatrix} $$
where $r(\varphi)$ is an arbitrary angle-dependent, positive radius. The unit disk is described by the radius function $r\equiv 1$. We now consider domains $\Omega_\varepsilon$ with a perturbed radius function of the form
$$r(\varphi,\varepsilon) = 1 + \sum\limits_{j \in\mathbb{Z}} (\varepsilon a_j + \varepsilon^2 b_j) e^{ij\varphi},$$
where we assume $a_j=\overline{a_{-j}}$ and $b_j=\overline{b_{-j}}$ so that the radius $r(\varphi,\varepsilon)$ is always real-valued. We also assume for simplicity that the number of non-zero coefficients $a_j$, $b_j$ is finite. This finiteness assumption may be replaced by suitable decay assumptions. Perturbations with only $a_{\pm1}\neq0$ correspond to first order to translations, under which $S_n$ is invariant. Local optimality is therefore understood modulo translations, cf.~the proof of Theorem \ref{thm:local_minimizer_disc}.

First, we notice that arbitrary perturbations will not preserve the area of $\Omega_\varepsilon$. The area of a domain with perturbed radius $r(\varphi, \varepsilon)$ is
\begin{align*}
|\Omega_\varepsilon|&= \frac{1}{2}\int_0^{2 \pi} r(\varphi,\varepsilon)^2\, d\varphi = \pi \bigg(1+2\varepsilon a_0+ \varepsilon^2 \bigg( 2b_0 + \sum\limits_{j \in\mathbb{Z}} |a_j|^2\bigg) + O(\varepsilon^3)\bigg).
\end{align*}
Since renormalized eigenvalues are invariant under domain scaling, we may assume
\begin{align}
 a_0 = 0, \qquad b_0 = - \frac{1}{2}\sum\limits_{j \in\mathbb{Z}} |a_j|^2   \label{eq:area_normalize_ajbj}
\end{align}
in the following, so that area is preserved up to second order in $\varepsilon$.

\subsection{Perturbation expansions for eigenfrequencies of the disk}

Let $\lambda_n(\Omega_\varepsilon) = k(\varepsilon)^2$ denote an eigenvalue of $\Omega_\varepsilon$. We now expand the perturbed eigenfrequency $k(\varepsilon)$ into 
$$k(\varepsilon) = k_0 + \varepsilon k_1 + \varepsilon^2 k_2 + O(\varepsilon^3) \qquad (\varepsilon \to 0).$$
Explicit expressions for $k_1 $ and $k_2$ are given by the next proposition.

\begin{proposition}[Wolf-Keller\cite{Wolf1994}] \label{prop:wk_lemma}
Let $\lambda_n(\mathbb{D}) = j_{mp}^2 $ and $k_0 = j_{mp}$. Then the first and second order perturbation expansion coefficients of $k(\varepsilon)=\sqrt{\lambda_n(\Omega_\varepsilon)}$ are given by the following expressions:
    \begin{enumerate}
        \item  If $m=0$, then
        \begin{align} \label{eq:k_1m0}
            \frac{k_1}{k_0} &= 0,\\ 
            \frac{k_2}{k_0} &= \sum\limits_{j\neq 0} \left(1 + k_0\frac{J_j'(k_0)}{J_j(k_0)}   \right) |a_{j}|^2.  \label{eq:k_2m0}
        \end{align}
        
        \item  If $m>0$, then
        
        \begin{align}  \label{eq:k_1mgtr0}
        \frac{k_1}{k_0} &= \begin{cases} 
        -|a_{2m}| , & \text{if } \lambda_n = \lambda_{n+1}, \qquad (I)\\
        +|a_{2m}|, &  \text{if } \lambda_n < \lambda_{n+1}, \qquad (II)
        \end{cases} \\
        \begin{split}
             \frac{k_2}{k_0} &= \bigg( \frac{1}{2} |a_{2m}|^2 + \sum\limits_{j \in\mathbb{Z}} |a_j|^2 + \sum\limits_{|j|\neq m} k_0\frac{J_j'(k_0)}{J_j(k_0)}  |a_{m-j}|^2 \bigg)  \\
&\qquad  + \frac{c_{-m}}{c_m} \bigg(-b_{2m}+ \frac{1}{2} \sum\limits_{j\in\mathbb{Z}} a_{m-j}a_{m+j} + \sum\limits_{|j|\neq m}   k_0\frac{J_j'(k_0)}{J_j(k_0)} a_{m+j}a_{m-j}\bigg)\\
&\qquad + \bigg(\frac{c_m'}{c_{-m}}a_{-2m} - \frac{c_{-m}'}{c_{m}}a_{2m} \bigg)
\end{split} \label{eq:k_2mgtr0}
        \end{align}
        where, if $a_{2m} \neq 0$,
        \begin{align*}
            c_m = \overline{c_{-m}} = \begin{cases}
                  e^{ i (\text{arg } a_{2m})/2} , & \text{if } \lambda_n = \lambda_{n+1}, \qquad (I)\\
                  e^{ i (\pi + \text{arg } a_{2m})/2}, &  \text{if } \lambda_n < \lambda_{n+1}. \qquad (II)
            \end{cases}
        \end{align*}
    \end{enumerate}
\end{proposition}

The coefficients $c_{\pm m}$, $c_{\pm m}'$ are in general some complex coefficients from the expansion \eqref{eq:disk_Cn_expansion} whose exact values become irrelevant in the next section through cancellations.

\begin{proof}[Proof of Proposition \ref{prop:wk_lemma}]
Within this proof, we work with the convention that $J_{-n} = J_n$, instead of the canonical relation $J_{-n} = (-1)^n J_n$. This will simplify calculations, as we don't have to bother about the sign of $J_{\pm n}$. The end result does not depend on this sign convention. The calculations follow Wolf and Keller \cite{Wolf1994} and are included here for sake of completeness. 

By virtue of the Helmholtz equation in $\mathbb{R}^2$, any eigenfunction of the perturbed domain can be written in the form
$$u(r,\varphi,\varepsilon) = \sum\limits_{n \in\mathbb{Z}} C_n(\varepsilon) J_n(k(\varepsilon)r) e^{in\varphi}$$
where 
\begin{align}
  C_n(\varepsilon) = \delta_{|n|,m}c_n + \varepsilon c_n' + \varepsilon^2 c_n'' + O(\varepsilon^3)  \label{eq:disk_Cn_expansion}
\end{align}
are complex coefficients. We have already used here that $\lambda_n(\Omega_0) = \lambda_{n}(\mathbb{D}) = j_{mp}^2 $ for some $m\in\mathbb{N}_0$ and some $p\in\mathbb{N}$. The closed form for the eigenfunctions on the unit disk implies that only the two coefficients $c_m$ and $c_{-m}$ are then non-zero. We further assume $C_{-n}(\varepsilon)=\overline{C_{n}(\varepsilon)}$ for any $n \in \mathbb{Z}$ so that the eigenfunctions are always real-valued. 

To prove the proposition, we use that any eigenfunction on the perturbed domain must satisfy the boundary condition $u=0$ on $\partial \Omega_\varepsilon$ or equivalently, $u(r(\varphi, \varepsilon),\varphi,\varepsilon)=0 $ for all $\varphi \in [0,2\pi)$.
Since $k(\varepsilon)r(\varphi, \varepsilon) = k_0 + (k(\varepsilon)-k_0)+ k(\varepsilon)(r(\varphi,\varepsilon)-1) = k_0 + O(\varepsilon)$ as $\varepsilon\to 0$, we obtain by Taylor expansion of the Bessel functions around the point $k_0$
\begin{align} \label{eq:near_disc_boundary_condition}
\begin{split}
u(r(\varphi,\varepsilon),\varphi,\varepsilon) = \sum\limits_{n \in\mathbb{Z}} C_n(\varepsilon) \bigg[ &J_n(k_0) + J_n'(k_0)[(k(\varepsilon)-k_0)+ k(\varepsilon)(r(\varphi, \varepsilon)-1)] \\
&+ \frac{1}{2} J_n''(k_0)[(k(\varepsilon)-k_0)+ k(\varepsilon)(r(\varphi, \varepsilon)-1)]^2+...\bigg] e^{in\varphi}=0.
\end{split}
\end{align} 
The coefficients $k_0$, $k_1$ and $k_2$ in the expansion of $k$ can now be extracted from \eqref{eq:near_disc_boundary_condition} by setting all orders of $\varepsilon$ to zero. At zeroth order, equation \eqref{eq:near_disc_boundary_condition} just becomes  $k_0 = j_{mp}$ for some $m\in\mathbb{N}_0$ and some $p\in\mathbb{N}$, as expected. For the next orders of $\varepsilon$, we note that
\begin{align*}
(k(\varepsilon)-k_0)+ k(\varepsilon)(r(\varphi,\varepsilon)-1) &=  \varepsilon\bigg( k_1 + k_0 \sum\limits_{j\in \mathbb{Z}} a_j  e^{ij\varphi}\bigg) \\
& \quad+ \varepsilon^2 \bigg(k_2 + k_1 \sum\limits_{j\in \mathbb{Z}} a_j  e^{ij\varphi} +k_0  \sum\limits_{j \in \mathbb{Z}}  b_j e^{ij\varphi}\bigg)+O(\varepsilon^3)
\end{align*}
as $\varepsilon \to 0$. \\

\textbf{Order $\varepsilon$, case $m>0$:} Assuming that $m>0$, collecting all terms of order $\varepsilon$ from equation \eqref{eq:near_disc_boundary_condition} gives
\begin{align}
J_m'(k_0)(c_m e^{im\varphi} +c_{-m} e^{-im\varphi})\bigg( k_1 + k_0 \sum\limits_{j \in \mathbb{Z}} a_j  e^{ij\varphi}\bigg)+\sum\limits_{n \in \mathbb{Z}}  J_n(k_0) c_n' e^{i n\varphi} = 0. \label{eq:disk_firstorder_eps1}
\end{align}
By reindexing, equation \eqref{eq:disk_firstorder_eps1} can be rewritten into a single sum over complex exponentials with corresponding coefficients. Since complex exponentials are linearly independent, equation \eqref{eq:disk_firstorder_eps1} can only hold for all $\varphi \in [0,2\pi )$ if all coefficients are zero. Coefficients of $e^{in\varphi}$, $|n|\neq m$, yield
\begin{align*}
k_0 J_m'(k_0) \big(a_j c_m |_{j+m = n}+a_jc_{-m} |_{j-m = n} \big) + J_n(k_0)c_n' = 0
\end{align*}
and thus
\begin{align}\label{near_disc_cn1}
c_n' = - k_0 \frac{J_m'(k_0)}{J_n(k_0)} (a_{n-m} c_m +a_{n+m}c_{-m}).
\end{align}
Setting the coefficient of $e^{im\varphi}$ in \eqref{eq:disk_firstorder_eps1} to zero yields
\begin{align*}
 J_m'(k_0) \big( k_1  c_m +k_0 (a_j c_m |_{j+m = m}+a_jc_{-m} |_{j-m = m}) \big) =0,
\end{align*}
as $J_m(k_0)=J_m(j_{mp})=0$. Since $a_0 = 0$ and $J_m'(k_0) = J_m'(j_{mp}) \neq 0$ for any $m\in\mathbb{N}_0$, $p \in\mathbb{N}$, this implies
\begin{align}\label{near_disc_k1_1}
 k_1 = -k_0 a_{2m} \frac{c_{-m}}{c_m}.
\end{align}
Similarly, setting the coefficient of $e^{-im\varphi}$ in \eqref{eq:disk_firstorder_eps1} to zero gives
% \begin{align*}
%  J_m'(k_0) \big( k_1 c_{-m} +k_0 (a_j c_m |_{j+m = -m}+a_jc_{-m} |_{j-m = -m})\big) =0
% \end{align*}
% or
% \begin{align*}
%  k_1  c_{-m} +k_0 (a_{-2m} c_m +a_{0}c_{-m}) =0,
% \end{align*}
% thus
\begin{align}\label{near_disc_k1_2}
 k_1 = -k_0 a_{-2m} \frac{c_{m}}{c_{-m}}.
\end{align}
Since $a_j=\overline{a_{-j}}$ and $c_m=\overline{c_{-m}}$, we infer from \eqref{near_disc_k1_1} and \eqref{near_disc_k1_2} that $k_1 = \overline{k_1}$, so $k_1$ is in fact a real number, though the sign of $k_1$ has yet to be determined. Since eigenfunctions may be scaled by arbitrary complex constants, we may assume now that $|c_m| = |c_{-m}| = 1$. If $a_{2m} \neq 0$, only the two cases
$$c_m = e^{ i (\text{arg } a_{2m})/2} \quad \text{(I)}\qquad \qquad\text{ or }\qquad \qquad c_m = e^{ i (\pi + \text{arg } a_{2m})/2}   \quad \text{(II)}, $$
up to a simultaneous change of sign, are compatible with \eqref{near_disc_k1_1} and \eqref{near_disc_k1_2}. In case (I), we get $k_1^{\text{(I)}} = - k_0 |a_{2m}|$, while in case (II), we get $k_1^{\text{(II)}} = + k_0 |a_{2m}|$. Eigenfrequencies are thus linearly split according to the two cases (I) and (II). This shows \eqref{eq:k_1mgtr0}.\\

\textbf{Order $\varepsilon$, case $m=0$:} Applying the same reasoning as in the previous case, one obtains
\begin{align*}
c_n' = - k_0 \frac{J_0'(k_0)}{J_n(k_0)} c_0 a_{n}, \qquad n \neq 0,
\end{align*}
and $k_1 = 0$. This shows \eqref{eq:k_1m0}.\\

\textbf{Order $\varepsilon^2$, case $m>0$:} At order $\varepsilon^2$, equation \eqref{eq:near_disc_boundary_condition} leads for $m>0$ to 
\begin{align} \label{eq:disk_firstorder_eps2}
\begin{split}
&(c_m e^{im\varphi}+c_{-m} e^{-{i}m\varphi})\Bigg[ J_m'(k_0)\bigg(k_2+k_1\sum\limits_{j \in\mathbb{Z}} a_j  {e}^{{i}j\varphi}+k_0\sum\limits_{j\in\mathbb{Z}} b_j  {e}^{{i}j\varphi}\bigg)\\
&\qquad \qquad\qquad + \frac{1}{2}J_m''(k_0)\bigg( k_1^2+2k_1 k_0\sum\limits_{j\in\mathbb{Z}} a_j  {e}^{{i}j\varphi} +k_0^2 \sum\limits_{j,k \in\mathbb{Z}} a_j a_k  {e}^{{i}(j+k)\varphi}\bigg)\Bigg]\\
&\qquad\qquad \qquad+ \sum\limits_{n \in\mathbb{Z}} c_n'J_n'(k_0) \bigg( k_1 + k_0\sum\limits_{j \in\mathbb{Z}} a_j  {e}^{{i}j\varphi}  \bigg){e}^{{i}n\varphi} + \sum\limits_{n\in\mathbb{Z}} c_n''J_n(k_0){e}^{{i}n\varphi}=0
\end{split}
\end{align}
Again, this equation can be rearranged into a sum of complex exponentials with matching coefficients which all have to be zero. Setting the coefficient of ${e}^{{i}m\varphi}$ in \eqref{eq:disk_firstorder_eps2} to zero gives
\begin{align*}
& J_m'(k_0)(c_mk_2+k_1  c_m a_j|_{j+m=m}  +k_0  c_m b_j|_{j+m=m} +k_1 c_{-m}a_j|_{j-m=m}  +k_0 c_{-m} b_j|_{j-m=m})\\
&+ \frac{1}{2}J_m''(k_0)\Bigg[c_m \bigg( k_1^2+2k_1 k_0 a_j|_{j+m=m} +k_0^2 \sum\limits_{j\in\mathbb{Z}} a_ja_k |_{j+k+m = m}\bigg)\\
& \qquad\qquad \qquad+c_{-m} \bigg( 2k_1 k_0 a_j|_{j-m=m} +k_0^2 \sum\limits_{j\in\mathbb{Z}} a_ja_k |_{j+k-m = m}\bigg)\Bigg]\\
&+  c_m'J_m'(k_0) k_1 + \sum\limits_{n\in\mathbb{Z}} c_n' J_n'(k_0) k_0 a_j |_{j+n=m}=0
\end{align*}
which can be rearranged to
\begin{align*}
& c_m\bigg( J_m'(k_0) (k_2 + k_1 a_{0} + k_0 b_{0}) + \frac{1}{2}J_m''(k_0)\bigg( k_1^2 + 2k_0k_1 a_{0}+k_0^2 \sum\limits_{j \in \mathbb{Z}} a_ja_{-j}\bigg) \bigg)\\
&+c_{-m}\bigg( J_m'(k_0) (k_1 a_{2m} + k_0 b_{2m}) + \frac{1}{2}J_m''(k_0)\bigg( 2k_0k_1 a_{2m}+k_0^2 \sum\limits_{j \in\mathbb{Z}} a_{j}a_{2m-j}\bigg) \bigg)\\
&+  c_m'J_m'(k_0) k_1 + \sum\limits_{j \in\mathbb{Z}} c_j' J_j'(k_0) k_0 a_{m-j}=0.
\end{align*}
Now insert $k_1$ from equations \eqref{near_disc_k1_1} and \eqref{near_disc_k1_2}, $c_j'$ for $|j|\neq m$ from equation \eqref{near_disc_cn1} and the identity $J_m''(k_0) = -J_m'(k_0)/k_0$ for $k_0 = j_{mp}$ which follows from Bessel's equation. This yields
\begin{align*}
& c_m\bigg( J_m'(k_0) (k_2  -k_0 a_{2m}a_{0} \frac{c_{-m}}{c_m}  + k_0 b_{0})  -\frac{J_m'(k_0)}{2k_0}\bigg( k_0^2 |a_{2m}|^2 - 2k_0 ^2a_{0}a_{2m} \frac{c_{-m}}{c_m} +k_0^2 \sum\limits_{j \in\mathbb{Z}} |a_j|^2\bigg) \bigg)\\
&+c_{-m}\bigg( J_m'(k_0) (-k_0 |a_{2m}|^2 \frac{c_{m}}{c_{-m}}  + k_0 b_{2m})  -\frac{J_m'(k_0)}{2k_0}\bigg( -2k_0^2 |a_{2m}|^2 \frac{c_{m}}{c_{-m}} +k_0^2 \sum\limits_{j \in\mathbb{Z}} a_{m-j}a_{m+j}\bigg) \bigg)\\
&+  c_m'J_m'(k_0) \bigg( -k_0 a_{-2m} \frac{c_{m}}{c_{-m}} \bigg) +c_m' J_m'(k_0) k_0 a_{0}+c_{-m}' J_m'(k_0) k_0 a_{2m}\\
&+ \sum\limits_{|j|\neq m} \bigg( - k_0 \frac{J_m'(k_0)}{J_j(k_0)} (a_{j-m} c_m +a_{j+m}c_{-m})\bigg) J_j'(k_0) k_0 a_{m-j}=0.
\end{align*}
Multiplying out, canceling and dividing by $J_m'(k_0)$ gives

\begin{align*}
& c_m k_2 + c_m k_0 b_0 - \frac{1}{2}c_m k_0 |a_{2m}|^2-\frac{1}{2}c_{m}k_0 \sum\limits_{j\in\mathbb{Z}} |a_j|^2\\
& + c_{-m} k_0 b_{2m}  - \frac{1}{2}c_{-m}k_0\sum\limits_{j\in\mathbb{Z}} a_{m-j}a_{m+j}\\
& - c_m' k_0 a_{-2m} \frac{c_m}{c_{-m}}+c_m'k_0 a_0+ c_{-m}'k_0 a_{2m}\\
& + \sum\limits_{|j|\neq m}  - k_0^2\frac{J_j'(k_0)}{J_j(k_0)} ( c_m |a_{m-j}|^2 +c_{-m}a_{m+j}a_{m-j})=0.
\end{align*}
Rearranging the last equation and inserting \eqref{eq:area_normalize_ajbj} gives \eqref{eq:k_2mgtr0}.
%\begin{align}
% \frac{k_2}{k_0} =& \left(- b_0 + \frac{1}{2} |a_{2m}|^2 +\frac{1}{2} \sum\limits_{j=-\infty}^\infty |a_j|^2 + \sum\limits_{|j|\neq m} k_0\frac{J_j'(k_0)}{J_j(k_0)}  |a_{m-j}|^2 \right)  \\
%& + \frac{c_{-m}}{c_m} \left(-b_{2m}+ \frac{1}{2} \sum\limits_{j=-\infty}^\infty a_{m-j}a_{m+j} + \sum\limits_{|j|\neq m}   k_0\frac{J_j'(k_0)}{J_j(k_0)} a_{m+j}a_{m-j}\right)\\
%& + \left(\frac{c_m'}{c_{-m}}a_{-2m} - \frac{c_{-m}'}{c_{m}}a_{2m} \right) - c_m'a_0 
%\end{align}
% \begin{align} \label{eq:near_disc_k2}
% \begin{split}
%  \frac{k_2}{k_0} =& \left( \frac{1}{2} |a_{2m}|^2 + \sum\limits_{j=-\infty}^\infty |a_j|^2 + \sum\limits_{|j|\neq m} k_0\frac{J_j'(k_0)}{J_j(k_0)}  |a_{m-j}|^2 \right)  \\
% & + \frac{c_{-m}}{c_m} \left(-b_{2m}+ \frac{1}{2} \sum\limits_{j=-\infty}^\infty a_{m-j}a_{m+j} + \sum\limits_{|j|\neq m}   k_0\frac{J_j'(k_0)}{J_j(k_0)} a_{m+j}a_{m-j}\right)\\
% & + \left(\frac{c_m'}{c_{-m}}a_{-2m} - \frac{c_{-m}'}{c_{m}}a_{2m} \right)
% \end{split} 
% \end{align}
Again note that $k_2$ depends on $c_m$, so the distinction between (I) and (II) will come into play and give different results depending on the case. \\

\textbf{Order $\varepsilon^2$, case $m=0$:} If $m=0$, then an analogous calculation gives
\begin{align} \label{eq:near_disc_k2_m0_simple}
 \frac{k_2}{k_0} =\sum\limits_{j\neq 0} \left(1 + k_0\frac{J_j'(k_0)}{J_j(k_0)}   \right) |a_{j}|^2.
\end{align}

\end{proof}

\subsection{Perturbation expansions for eigenvalue sums of the disk}

Obtaining expansions for eigenvalue sums are now a matter of combining the expansions for individual eigenvalues. We have the following perturbation expansions.

\begin{proposition} \label{prop:prop_Sn_exp}
Let $\lambda_n(\mathbb{D})=j_{mp}^2$.
\begin{enumerate}
    \item If the $n$-th eigenvalue of the unit disk is the first instance of a double eigenvalue, i.e.~$\lambda_n(\mathbb{D})=\lambda_{n+1}(\mathbb{D})$, then
    \begin{align} \label{eq:prop_Sn_exp_case1}
        S_n(\Omega_\varepsilon) = S_n(\mathbb{D}) - \varepsilon \cdot 2 \pi j_{mp}^2 |a_{2m}|+O(\varepsilon^2) \qquad (\varepsilon\searrow 0).
    \end{align}
    \item If the $n$-th eigenvalue of the unit disk is a single eigenvalue or the second instance of a double eigenvalue, i.e.~$\lambda_{n}(\mathbb{D}) < \lambda_{n+1}(\mathbb{D}) $, then
    \begin{align} \label{eq:prop_Sn_exp_case2}
    \begin{split}
S_n(\Omega_\varepsilon) =  S_n(\mathbb{D}) &+ \varepsilon^2 \cdot \Bigg[ \sum\limits_{j_{0p}^2 \leq \lambda_n} \bigg( 2\pi j_{0p}^2 \sum\limits_{j\neq 0} \bigg(1 + j_{0p}\frac{J_j'(j_{0p})}{J_j(j_{0p})}   \bigg) |a_{j}|^2  \bigg)\\
&+\sum\limits_{\substack{j_{mp}^2 \leq \lambda_n, \\ m>0}} 4\pi j_{mp}^2 \bigg(  |a_{2m}|^2 + \sum\limits_{j\in\mathbb{Z}} |a_j|^2 + \sum\limits_{j\neq 0,2m} j_{mp}\frac{J_{m-j}'(j_{mp})}{J_{m-j}(j_{mp})}  |a_{j}|^2 \bigg)\Bigg] \\
&+ O(\varepsilon^3)  \qquad\qquad (\varepsilon\to 0).
\end{split}
\end{align}
\end{enumerate}
\end{proposition}

\begin{proof}
Every renormalized eigenvalue satisfies
\begin{align} \label{eq:single_eig_val_expansion_2D}
|\Omega_\varepsilon|\lambda_k(\Omega_\varepsilon) &= \pi k_{0}^2 \left(1 + 2\varepsilon \frac{k_1}{k_0} + \varepsilon^2 \left(2 \frac{k_2}{k_0} + \frac{k_1^2}{k_0^2}\right) + O(\varepsilon^3) \right)
\end{align}
with corresponding coefficients $k_0$, $k_1$ and $k_2$ from Proposition \ref{prop:wk_lemma} depending on the index $k$. We now split $S_n(\Omega_\varepsilon)=\sum_{k=1}^n |\Omega_\varepsilon|\lambda_k(\Omega_\varepsilon)$ into sums over single and double eigenvalues, i.e.~
\begin{align} \label{eq:S_n2d_splitting}
    S_n(\Omega_\varepsilon) = \sum\limits_{\substack{\lambda_k(\mathbb{D})=j_{0p}^2} \leq \lambda_n(\mathbb{D}) } |\Omega_\varepsilon|\lambda_k(\Omega_\varepsilon) + \sum\limits_{\substack{\lambda_k(\mathbb{D})=j_{mp}^2 \leq \lambda_n(\mathbb{D}), \\ m>0}} |\Omega_\varepsilon|\lambda_k(\Omega_\varepsilon) .
\end{align}
By Proposition \ref{prop:wk_lemma},
\begin{align} \label{eq:Alambda_m0}
|\Omega_\varepsilon|\lambda_k(\Omega_\varepsilon) = \pi k_0^2 \bigg(1 + \varepsilon^2\cdot 2\sum\limits_{j\neq 0} \bigg(1 + k_0\frac{J_j'(k_0)}{J_j(k_0)}   \bigg) |a_{j}|^2 + O(\varepsilon^3) \bigg),
\end{align}
thus, the first sum in \eqref{eq:S_n2d_splitting} has no first order contributions. For the second sum in \eqref{eq:S_n2d_splitting}, we always group the two instances of a double eigenvalue in pairs and then sum the perturbation expansions of these pairs. For indices $k$ with $\lambda_{k-1}(\mathbb{D})=\lambda_{k}(\mathbb{D}) \leq \lambda_n(\mathbb{D})$, we can identify $\lambda_{k-1}(\mathbb{D})$ as a case (I) eigenvalue and $\lambda_{k}(\mathbb{D})$ as a case (II) eigenvalue in Proposition \ref{prop:wk_lemma}. Both first order perturbations $k_1$ from \eqref{eq:k_1mgtr0} cancel when summed. This leads to
\begin{align*}
|\Omega_\varepsilon|\left(\lambda_{k-1}(\Omega_\varepsilon)+\lambda_{k}(\Omega_\varepsilon)\right) &= 2 \pi k_0^2 \left(1 + \varepsilon^2 \left( \frac{k_2^{\text{(I)}}}{k_0} +\frac{k_2^{\text{(II)}}}{k_0} + |a_{2m}|^2\right) + O(\varepsilon^3) \right).
\end{align*}
Since
\begin{align} \label{eq:disk_cmI-II_relation}
\frac{c_{-m}^{\text{(I)}}}{c_{m}^{\text{(I)}}} = - \frac{c_{-m}^{\text{(II)}}}{c_{m}^{\text{(II)}}},
\end{align} 
we can see that the entire second bracket in \eqref{eq:k_2mgtr0} with prefactor $c_{-m}^{\text{(I)}}/c_{m}^{\text{(I)}} $ in case (I) exactly cancels the second bracket in \eqref{eq:k_2mgtr0} with prefactor $c_{-m}^{\text{(II)}}/c_{m}^{\text{(II)}}$ in case (II)\footnote{If $a_{2m}=0$, then the phase of $c_m$ is not fixed by \eqref{near_disc_k1_1} and \eqref{near_disc_k1_2}. Relation \eqref{eq:disk_cmI-II_relation} nevertheless holds: eigenfunctions of $\lambda_{k-1}(\Omega_\varepsilon)$ and $\lambda_k(\Omega_\varepsilon)$ can be chosen orthogonal for every $\varepsilon$, hence so are their limits as $\varepsilon\to 0$. With $|c_m^{\text{(I)}}|=|c_m^{\text{(II)}}|=1$, this forces $c_m^{\text{(II)}}=\pm i\, c_m^{\text{(I)}}$, which is \eqref{eq:disk_cmI-II_relation}.}. Furthermore, the third bracket in \eqref{eq:k_2mgtr0} is
\begin{align*}
\bigg(\frac{c_m'}{c_{-m}}a_{-2m} - \frac{c_{-m}'}{c_{m}}a_{2m} \bigg),
\end{align*}
which is always purely imaginary because $a_{2m}=\overline{a_{-2m}}$, $c_m=\overline{c_{-m}}$ and $c_m'=\overline{c_{-m}'}$. On the other hand the first bracket in \eqref{eq:k_2mgtr0} is always purely real, so in order to keep $k_2^{\text{(I)}}/k_0 +k_2^{\text{(II)}}/k_0 $ real-valued, we must have cancellation of all imaginary terms, i.e.
\begin{align*} 
\bigg(\frac{c_m'{}^{\text{(I)}}}{c_{-m}^{\text{(I)}}}a_{-2m} - \frac{c_{-m}'{}^{\text{(I)}}}{c_{m}^{\text{(I)}}}a_{2m} \bigg) + \bigg(\frac{c_m'{}^{\text{(II)}}}{c_{-m}^{\text{(II)}}}a_{-2m} - \frac{c_{-m}'{}^{\text{(II)}}}{c_{m}^{\text{(II)}}}a_{2m} \bigg) = 0.
\end{align*}
Therefore,
\begin{align*} 
 \frac{k_2^{\text{(I)}}}{k_0} +\frac{k_2^{\text{(II)}}}{k_0} =& |a_{2m}|^2 + 2\sum\limits_{j \in\mathbb{Z}} |a_j|^2 + 2\sum\limits_{|j|\neq m} k_0\frac{J_j'(k_0)}{J_j(k_0)}  |a_{m-j}|^2 
\end{align*}
and 
\begin{align}\label{eq:Alambd_sum_m_greater_0}
\begin{split}
&|\Omega_\varepsilon|\left(\lambda_{k-1}(\Omega_\varepsilon)+\lambda_{k}(\Omega_\varepsilon)\right) \\%=\; & 2 \pi k_0^2 \bigg(1 + \varepsilon^2 \cdot 2 \bigg( |a_{2m}|^2 + \sum\limits_{j\in\mathbb{Z}} |a_j|^2 + \sum\limits_{|j|\neq m} k_0\frac{J_j'(k_0)}{J_j(k_0)}  |a_{m-j}|^2\bigg) + O(\varepsilon^3) \bigg)\\
 =\; &2 \pi k_0^2 \bigg(1 + \varepsilon^2 \cdot 2 \bigg( |a_{2m}|^2 + \sum\limits_{j \in \mathbb{Z}} |a_j|^2 + \sum\limits_{j\neq 0,2m} k_0\frac{J_{m-j}'(k_0)}{J_{m-j}(k_0)}  |a_{j}|^2\bigg) + O(\varepsilon^3) \bigg).
\end{split}
\end{align}
Now consider the case that $\lambda_{n}(\mathbb{D})$ is a single eigenvalue or the second instance of a double eigenvalue. Inserting \eqref{eq:Alambda_m0} for all single eigenvalues and \eqref{eq:Alambd_sum_m_greater_0} for all pairs of double eigenvalue instances into \eqref{eq:S_n2d_splitting} we obtain \eqref{eq:prop_Sn_exp_case2}. In the other case where $\lambda_{n}(\mathbb{D})$ is the first instance of a double eigenvalue, $S_{n-1}(\Omega_\varepsilon)$ falls under the previous case and Proposition \ref{prop:wk_lemma} together with \eqref{eq:single_eig_val_expansion_2D} gives $|\Omega_\varepsilon|\lambda_n(\Omega_\varepsilon)=\pi \lambda_n(\mathbb{D}) - \varepsilon \cdot 2\pi j_{mp}^2  |a_{2m}| + O(\varepsilon^2)$ as $\varepsilon \searrow 0$. Thus,
\begin{align*}
    S_n(\Omega_\varepsilon) &=  S_{n-1 }(\Omega_\varepsilon) + |\Omega_\varepsilon|\lambda_n(\Omega_\varepsilon) \\
    &=S_{n-1 }(\mathbb{D}) + O(\varepsilon^2) + \pi \lambda_n(\mathbb{D}) - \varepsilon \cdot 2\pi j_{mp}^2  |a_{2m}| + O(\varepsilon^2) \\
    &=S_{n }(\mathbb{D})  - \varepsilon \cdot 2\pi j_{mp}^2  |a_{2m}| + O(\varepsilon^2) \qquad (\varepsilon \searrow 0).
\end{align*}   
\end{proof}

Theorem \ref{thm:local_minimizer_disc} now readily follows from the previous proposition.

\begin{proof}[Proof of Theorem \ref{thm:local_minimizer_disc}]
Case (1) of the theorem follows immediately from case (1) of Proposition \ref{prop:prop_Sn_exp} which shows that the disk is not a local minimizer: by choosing a perturbation of the disk with $a_{2m} \neq 0$ one can deform the disk to a domain $\Omega_\varepsilon$ with $S_n(\Omega_\varepsilon) < S_n(\mathbb{D})$ for small enough $\varepsilon$. 

In case (2), we can rephrase \eqref{eq:prop_Sn_exp_case2} as a weighted sum of the perturbation coefficients $|a_j|^2$. Since $a_0 = 0$ by assumption, we can write
\begin{align*}
S_n(\Omega_\varepsilon) =  S_n(\mathbb{D}) + \varepsilon^2 \cdot  \sum\limits_{j \neq 0} d_j |a_j|^2 + O(\varepsilon^3) \qquad (\varepsilon\to 0)
\end{align*}
with coefficients $d_j$ as stated in \eqref{eq:local_minimizer_disc_case2_dj}. Note that $|a_j|^2 = |a_{-j}|^2$ and $S_n(\Omega_\varepsilon)$ is invariant under translations of the domain. Therefore, we arrive at
\begin{align*}
S_n(\Omega_\varepsilon) =  S_n(\mathbb{D}) + \varepsilon^2 \cdot  \sum\limits_{j \geq 1} (d_j+d_{-j}) |a_j|^2 + O(\varepsilon^3) \qquad (\varepsilon\to 0).
\end{align*}
We claim that $d_1 + d_{-1} = 0$ for any $n$. This is because $j=1$ corresponds to infinitesimal translations of the disk. Indeed, by definition,
    \begin{align*} 
d_1 + d_{-1} &=   \sum\limits_{j_{0p}^2 \leq \lambda_n(\mathbb{D})} 2\pi j_{0p}^2 \left(2 + 2j_{0p}\frac{J_1'(j_{0p})}{J_1(j_{0p})}   \right)  \\
& \qquad+ \sum\limits_{\substack{j_{mp}^2 \leq \lambda_n(\mathbb{D}), \\ m>0}} 4\pi j_{mp}^2 \left( 2 +  j_{mp}\frac{J_{m-1}'(j_{mp})}{J_{m-1}(j_{mp})} +j_{mp}\frac{J_{m+1}'(j_{mp})}{J_{m+1}(j_{mp})}  \right).
\end{align*}
From the recurrence relations $J_\nu'(x) = J_{\nu-1}(x) - \frac{\nu}{x}J_\nu(x) = \frac{\nu}{x}J_\nu(x) - J_{\nu+1}(x)$ for $\nu \in \mathbb{R}$, one can show that
\begin{align*}
    j_{0p}\frac{J_1'(j_{0p})}{J_1(j_{0p})}  &= -1 , \\
    j_{mp}\frac{J_{m-1}'(j_{mp})}{J_{m-1}(j_{mp})} &= m-1, \\
    j_{mp}\frac{J_{m+1}'(j_{mp})}{J_{m+1}(j_{mp})} &= - (m+1),
\end{align*}
from which the claim follows.

Hence, the disk is a weak local minimizer of $S_n(\Omega)$ if all second order perturbation coefficients $d_j+d_{-j}$, $j\geq 2$, are positive. If $d_{j}+d_{-j}$ is negative for some $j\geq 2$, then the disk is not a local minimizer but a saddle point, because perturbing the circular boundary with $a_{j}\neq 0$ can decrease $S_n(\Omega)$ quadratically for small enough $\varepsilon$. 
\end{proof}

The next lemma shows that one cannot decrease $S_n(\Omega)$ with highly oscillating boundary perturbations and supplies this with a sufficient quantitative bound to their frequency.

\begin{lemma} \label{lem:dj_pos_2mmax}
Let $n\in\mathbb{N}$ and $m_{max}(n) := \max\lbrace m : j_{m1}^2 \leq \lambda_n(\mathbb{D})\rbrace $. Then
\begin{align*}
    d_j+d_{-j}>0  
\end{align*}
for all  $j\geq \big(1+\sqrt{4/3} \big) m_{max}(n)+5$.
\end{lemma}

This means that in order to numerically check the optimality condition (2) from Theorem \ref{thm:local_minimizer_disc}, it is sufficient to compute perturbation coefficients with $j < \big(1+\sqrt{4/3} \big) m_{max}(n)+5$, which is a finite computation. Furthermore, one can show that $m_{max}(n) = O(\sqrt{n})$ as $n \to \infty$. The number of second order perturbation coefficients needed to decide whether the disk is a local minimizer of $S_n(\Omega)$ is thus relatively small compared to $n$ and in fact computationally quite inexpensive.

\begin{proof}[Proof of Lemma \ref{lem:dj_pos_2mmax}]
It is sufficient to show that $J_{m \pm j}'(j_{mp})/J_{m \pm j}(j_{mp})>0$ for all $m\in \mathbb{N}_0$, $p\in\mathbb{N}$ such that $j_{mp}^2 \leq \lambda_n(\mathbb{D})$ and $j\geq \big(1+\sqrt{4/3} \big) m_{max}(n)+5$. For this purpose, recall two bounds for zeros of Bessel functions and their derivatives. By \cite[pp. 485-487]{Watson1995}, we have
\begin{align} \label{eq:besseljzero_m1_ineq}
j_{m1} < \sqrt{\frac{4}{3}(m+1)(m+5)},
\end{align}
and, denoting the zeros of $J_m'$ by $j_{m p}'$,
\begin{align} \label{eq:besseljzero_derivative_n1_ineq}
j_{m 1}' > m.
\end{align}
Furthermore, every $J_m$, $m\neq 0$, is positive and strictly increasing up to the first zero of its derivative $J_m'$, i.e. $J_m'(x)/J_m(x)>0$ for $0<x<j_{m 1}' $. So let $m\in \mathbb{N}_0$, $p\in\mathbb{N}$ such that $j_{mp}^2 \leq \lambda_n(\mathbb{D})$ and $j\geq \big(1+\sqrt{4/3} \big) m_{max}(n)+5$. Then,
\begin{align*}
    j_{mp} \leq j_{m_{max}(n)+1,1} &< \sqrt{\frac{4}{3}(m_{max}(n)+2)(m_{max}(n)+6)} \\
    &<\sqrt{\frac{4}{3}} (m_{max}(n)+4) \\
    &< j - m_{max}(n) \\
    &\leq j \pm m \\
    &< j_{j \pm m, 1}'
\end{align*}
and hence $J_{m \pm j}'(j_{mp})/J_{m \pm j}(j_{mp})>0$ as $m\pm j \neq 0$.
\end{proof}

\section{Second order perturbation expansions for the ball}

Parallel to the analysis in two dimensions, we derive a second order perturbation expansion for the eigenfrequencies $k = \sqrt{\lambda}$ for the Dirichlet Laplacian and then square and sum the results.

\subsection{Three-dimensional perturbation framework} \label{subsec:threedim_framework}

The eigenvalues of the Dirichlet Laplacian on the unit ball are given by $\lambda_{n}(\mathbb{B}) = j_{l+1/2,p}^2$ for some $l\in\mathbb{N}_0$, $p\in\mathbb{N}$. If $l=0$, the eigenvalues $\lambda_{n}(\mathbb{B}) = j_{1/2,p}^2$ are simple, while if $l>0$ the eigenvalues $\lambda_{n}(\mathbb{B}) = j_{l+1/2,p}^2$ are $(2l+1)$-fold degenerate. The eigenspace of $\lambda_{n}(\mathbb{B}) = j_{l+1/2,p}^2$ is spanned by $(2l+1)$ eigenfunctions of the form
\begin{align*}
u_{lm}(r,\theta,\varphi)= j_l(j_{l+1/2,p}r) Y_{l}^{m}(\theta, \varphi), \qquad \qquad m \in\mathbb{Z}, \; -l\leq m \leq l,
\end{align*}
where 
$$j_l(x)= \sqrt{\frac{\pi}{2x}} J_{l+1/2}(x)$$
are spherical Bessel functions and $Y_{l}^{m}$ are spherical harmonics. We remark that since $j_l(x)=\sqrt{\pi/(2x)}\,J_{l+1/2}(x)$, the fact that $j_l$ and $j_{l'}$, $l\neq l'$, have no common positive zeros is again due to Siegel's theorem \cite{Siegel1929}. This guarantees that all denominators in Theorem \ref{thm:local_minimizer_ball3d} are non-zero.

We consider near-spherical domains that arise from perturbing the boundary of the unit ball. Let $\Omega$ be a star-shaped domain whose boundary is the surface
$$s: [0,\pi] \times [0,2 \pi) \rightarrow \mathbb{R}^3, \quad s(\theta,\varphi)=r(\theta,\varphi)\begin{pmatrix} \sin \theta \cos \varphi \\  \sin \theta \sin \varphi \\  \cos \theta \end{pmatrix} $$
where $r(\theta,\varphi)$ is an arbitrary angle-dependent radius. The unit ball in $\mathbb{R}^3$ is described by the radius function $r \equiv 1$. We consider domains $\Omega_\varepsilon$ with a perturbed radius function of the form
\begin{align*}
r(\theta,\varphi, \varepsilon) = 1 + \sum\limits_{l,m} (\varepsilon a_{l,m} + \varepsilon^2 b_{l,m}) Y_{l}^{m}(\theta, \varphi)
\end{align*}
where $Y_{l}^{m}=(-1)^m \overline{Y_{l}^{-m}}$ are the usual spherical harmonics and we assume $a_{l,-m}=(-1)^m \overline{a_{l,m}}$ and $b_{l,-m}=(-1)^m \overline{b_{l,m}}$ so that the radius $r$ is real-valued. We may assume again for simplicity that the number of non-zero coefficients $a_{l,m}$, $b_{l,m}$ is finite. Perturbations with only $a_{1,M}\neq 0$ correspond to first order to translations, under which $S_n$ is invariant. We understand local optimality again modulo translations, cf.~the proof of Theorem \ref{thm:local_minimizer_ball3d}. Here and in the following, summations over indices of spherical harmonics are implicit over $l\in \mathbb{N}_0$, $-l \leq m \leq l$. 

The volume of a perturbed domain with radius $r(\theta,\varphi, \varepsilon)$ is
\begin{align*}
|\Omega_\varepsilon|&= \frac{1}{3}\int_0^{\pi} \int_0^{2 \pi} r(\theta,\varphi,\varepsilon)^3 \sin \theta \, d\varphi \, d\theta \\
% &= \frac{1}{3}\int_0^{\pi} \int_0^{2 \pi} \Bigg( 1 + 3\sum\limits_{l=0}^{\infty} \sum\limits_{m=-l}^{l} (\varepsilon a_{l,m} + \varepsilon^2 b_{l,m}) Y_{l}^{m}(\theta, \varphi) \\
% & \quad + 3 \sum\limits_{l=0}^{\infty} \sum\limits_{m=-l}^{l} \sum\limits_{l'=0}^{\infty} \sum\limits_{m'=-l'}^{l'} (\varepsilon a_{l,m} + \varepsilon^2 b_{l,m}) (\varepsilon a_{l',m'} + \varepsilon^2 b_{l',m'})Y_{l}^{m}(\theta, \varphi) Y_{l'}^{m'}(\theta, \varphi) +O(\varepsilon^3)\Bigg)\sin \theta \, d\varphi \, d\theta \\
&=\frac{4\pi}{3}  + \sum\limits_{l,m} (\varepsilon a_{l,m} + \varepsilon^2 b_{l,m})\int_0^{\pi} \int_0^{2 \pi} Y_{l}^{m}(\theta, \varphi)\sin \theta \, d\varphi \, d\theta\\
& \quad +  \varepsilon^2\sum\limits_{l,m}  \sum\limits_{l',m'}   a_{l,m}  a_{l',m'} (-1)^{m'}\int_0^{\pi} \int_0^{2 \pi} Y_{l}^{m}(\theta, \varphi) \overline{Y_{l'}^{-m'}(\theta, \varphi)} \sin \theta\, d\varphi \, d\theta +O(\varepsilon^3)
\end{align*}
By orthonormality of the spherical harmonics, we have
\begin{align*}
\int_0^{\pi} \int_0^{2 \pi} Y_{l}^{m}(\theta, \varphi) \overline{Y_{l'}^{m'}(\theta, \varphi)} \sin \theta \, d\varphi \, d\theta = \delta_{l,l'} \delta_{m,m'}
\end{align*}
and because $Y_{0}^{0}(\theta, \varphi) \equiv \sqrt{1/(4\pi)}$, we also have
\begin{align*}
\int_0^{\pi} \int_0^{2 \pi} Y_{l}^{m}(\theta, \varphi)\sin \theta \, d\varphi \, d\theta =\sqrt{4\pi} \delta_{l,0} \delta_{m,0}.
\end{align*}
Therefore,
\begin{align*}
|\Omega_\varepsilon|=\frac{4\pi}{3} + \varepsilon \,\sqrt{4\pi} a_{0,0} + \varepsilon^2 \bigg( \sqrt{4\pi} b_{0,0} +\sum\limits_{l,m}  |a_{l,m}|^2 \bigg) +O(\varepsilon^3).
\end{align*}
Due to scaling invariance of renormalized eigenvalues, we may choose
\begin{align} \label{eq:a00_b00}
  a_{0,0} = 0, \qquad b_{0,0} = - \sqrt{\frac{1}{4\pi}}\sum\limits_{l,m}   |a_{l,m}|^2   
\end{align}
in the following to guarantee a constant volume up to second order in $\varepsilon$. Finally note that if $|\Omega_\varepsilon|=|\mathbb{B}|+O(\varepsilon^3)$, then also  $|\Omega_\varepsilon|^{2/3}=|\mathbb{B}|^{2/3}+O(\varepsilon^3)$.

\subsection{Products of spherical harmonics and Clebsch-Gordan coefficients}

In what follows, we need to use several properties of Clebsch-Gordan coefficients. This is because they appear when one expresses products of spherical harmonics in the basis of spherical harmonics. The representation given by the next lemma is a tool used repeatedly for the calculation of perturbation expansions of eigenfrequencies in the next section. 

\begin{lemma}[{\cite[Section 5.6.2]{Varshalovich1988}}] \label{lem:sphericalharmonics_exp}
It holds
\begin{align*}
Y_{l}^{m}(\theta, \varphi)Y_{l'}^{m'}(\theta, \varphi) = \sum\limits_{L, M} d_{l,m,l',m'}^{L,M} Y_{L}^{M}(\theta, \varphi)
\end{align*}
where 
\begin{align} \label{eq:dlmlmLM_definition}
d_{l,m,l',m'}^{L,M} = \sqrt{\frac{(2l+1)(2l'+1)}{4\pi(2L+1)}} \langle l 0 l' 0|L 0\rangle \langle l m l' m'|L M\rangle
\end{align}
and $\langle l m l' m'|L M\rangle$ are Clebsch-Gordan coefficients. 
\end{lemma}

Clebsch-Gordan coefficients satisfy the following properties.

\begin{lemma}[{\cite[Chapter 8]{Varshalovich1988}}] \label{lem:cg_coeff} \,

% The Clebsch-Gordan coefficients are given by
% \begin{align*}
%  \langle l m l' m'|L M\rangle &= \delta_{m+m',M} \sqrt{\frac{(2L+1)(L+l-l')!(L-l+l')!(l+l'-L)!}{(l+l'+L+1)!}}\\
%  & \qquad \cdot \sqrt{(L-M)!(L+M)!(l-m)!(l+m)!(l'-m')!(l'+m')!}\\
%  & \qquad \cdot \sum\limits_{k} \frac{(-1)^k}{k! (l+l'-L-k)!(l-m-k)!(l'+m'-k)!(L-l'+m+k)!(L-l-m'+k)!}
% \end{align*}
% where the sum goes over all integers $k$ such that the argument of every factorial is non-negative. 
\begin{enumerate}
\item All Clebsch-Gordan coefficients $\langle l m l' m'|L M\rangle$ are real.
\item Selection rules: 
\begin{align*}
 \langle l m l' m'|L M\rangle = 0 \qquad \text{if } m+m' \neq M \text{ or } L<|l-l'| \text{ or } L>l+l' 
\end{align*}
\item Symmetry under change of sign:
\begin{align*}
 \langle l m l' m'|L M\rangle = (-1)^{l+l'-L}\langle l, -m ,l', -m'|L, -M\rangle
\end{align*}
\item Permutation symmetry (1):
\begin{align*}
 \langle l m l' m'|L M\rangle = (-1)^{l+l'-L}\langle l' m' l m|L M\rangle
\end{align*}
\item Permutation symmetry (2):
\begin{align*}
 \langle l m l' m'|L M\rangle &= (-1)^{l-m}\sqrt{\frac{2L+1}{2l'+1}}\langle l, m, L,-M |l',-m'\rangle\\
&= (-1)^{l'+m'}\sqrt{\frac{2L+1}{2l+1}}\langle L, -M, l',m' |l,-m\rangle\\
&= (-1)^{l-m}\sqrt{\frac{2L+1}{2l'+1}}\langle L, M, l,-m |l',m'\rangle\\
&= (-1)^{l'+m'}\sqrt{\frac{2L+1}{2l+1}}\langle l', -m', L,M |l,m\rangle
\end{align*}
\item Special case ($l'=m'=0$):
\begin{align*}
 \langle l, m, 0,0|L,M\rangle  = \delta_{l,L} \delta_{m,M}
\end{align*}
\item Summation identity:
\begin{align*}
 \sum\limits_{m=-l}^{l }\langle l, m, l', 0|l, m\rangle = (2l+1)\delta_{l',0} 
\end{align*}
\item Orthogonality:
\begin{align*}
\sum_{m=-l}^l \sum_{m'=-l'}^{l'} \langle l, m, l',m'|L_1,M_1\rangle \langle l, m, l',m'|L_2,M_2\rangle  = \delta_{L_1,L_2} \delta_{M_1,M_2}
\end{align*}
\end{enumerate}
\end{lemma}

These immediately imply corresponding properties and rules for the coefficients $d_{l,m,l',m'}^{L,M}$ of the expansion of a product of two spherical harmonics.

\begin{corollary} \label{cor:sphericalharmonics_exp_coeff_properties}
Let $d_{l,m,l',m'}^{L,M}$ be defined by \eqref{eq:dlmlmLM_definition}. Then:
\begin{enumerate}
\item All coefficients $d_{l,m,l',m'}^{L,M}$ are real.
\item Selection rules: 
\begin{align*}
 d_{l,m,l',m'}^{L,M} = 0 \qquad \text{if } m+m' \neq M \text{ or } L<|l-l'| \text{ or } L>l+l'
\end{align*}
\item Symmetry under change of sign:
\begin{align*}
d_{l,m,l',m'}^{L,M} = d_{l,-m,l',-m'}^{L,-M}
\end{align*}
\item Permutation symmetry (1):
\begin{align*}
d_{l,m,l',m'}^{L,M} = d_{l',m',l,m}^{L,M}
\end{align*}
\item Permutation symmetry (2):
\begin{align*}
d_{l,m,l',m'}^{L,M} &= (-1)^{-m}d_{l,m,L,-M}^{l',-m'}\\
&= (-1)^{m'}d_{ L, -M, l',m'}^{l,-m}\\
&= (-1)^{-m}d_{ L, M, l,-m }^{l',m'}\\
&= (-1)^{m'}d_{ l', -m', L,M }^{l,m}
\end{align*}
\item Special case ($l'=m'=0$):
$$d_{l,m,0,0}^{L,M} = \sqrt{\frac{1}{4 \pi}} \delta_{l,L}  \delta_{m,M}$$
\item Summation identity:
\begin{align*}
 \sum\limits_{m=-l}^{l }d_{l, m, l', 0}^{l, m} = \sqrt{\frac{1}{4 \pi}}(2l+1)\delta_{l',0} 
\end{align*}
\item Orthogonality:
\begin{align*}
\sum_{m=-l}^l \sum_{m'=-l'}^{l'} d_{ l, m, l',m'}^{L_1,M_1}  d_{ l, m, l',m'}^{L_2,M_2} = \frac{(2l+1)(2l'+1)}{4\pi (2L_1+1)}  \left(\langle l, 0, l',0|L_1,0\rangle \right)^2 \delta_{L_1,L_2} \delta_{M_1,M_2} 
\end{align*}
\end{enumerate}
\end{corollary}

Another property can be inferred from the symmetry under change of sign.

\begin{corollary} \label{cor:d_coeff_lsum_odd}
If $l+l'-L$ is odd, then
\begin{align*}
 d_{l,m,l',m'}^{L,M} = 0 .
\end{align*}
\end{corollary}
\begin{proof}
From the symmetry under change of sign, we get
$\langle l 0 l' 0|L 0\rangle = (-1)^{l+l'-L}\langle l 0 l' 0|L 0\rangle$. If $l+l'-L$ is odd, then $\langle l 0 l' 0|L 0\rangle=-\langle l 0 l' 0|L 0\rangle =0 $. Thus $d_{l,m,l',m'}^{L,M} = 0$.
\end{proof}

\subsection{Perturbation expansions for eigenfrequencies of the ball: first order}

Let $\lambda(\Omega_\varepsilon) = k(\varepsilon)^2$ be an eigenvalue of the near-spherical domain. Just as in two dimensions, we expand $k(\varepsilon)$ into 
\begin{align*}
    k(\varepsilon) = k_0 + \varepsilon k_1 + \varepsilon^2 k_2 + O(\varepsilon^3) \qquad (\varepsilon \to 0).
\end{align*}

\begin{proposition}   \label{prop:ball_first_order}
Let $\lambda_n(\mathbb{B}) = j_{l+1/2,p}^2 $ and $k_0 = j_{l+1/2,p}$. There exists an $n'\in \mathbb{N}$ such that $\lambda_{n'+\ell}(\mathbb{B}) = j_{l+1/2,p}^2 $ for all $-l\leq \ell\leq l$. Set $j = n-n'$ and denote the eigenvalues of the Hermitian matrix $A \in \mathbb{C}^{(2l+1)\times (2l+1)}$ defined by
\begin{align} \label{eq:k1_matrix_setup}
    A_{M,m} = -k_0\sum_{l',m'} a_{l',m'} d_{l,m,l',m'}^{l,M}, \qquad -l\leq M,m \leq l, 
\end{align}
in increasing order by $k_{1}^{(-l)}, ... , k_1^{(l)}$. Then the first order perturbation expansion coefficient of $k(\varepsilon)=\sqrt{\lambda_n(\Omega_\varepsilon)}$ is given by $k_1 = k_{1}^{(j)}$. Furthermore, it holds
\begin{align} \label{eq:k1_ball_trace}
    \sum_{\ell=-l}^l k_{1}^{(\ell)} = 0.
\end{align}

\end{proposition}

\begin{proof}
By solving the Helmholtz equation in spherical coordinates, eigenfunctions on the perturbed domain have the form 
\begin{align*}
u(r,\theta,\varphi,\varepsilon) = \sum\limits_{L,M}  C_{L,M}(\varepsilon) j_{L}(k(\varepsilon) r) Y_{L}^{M}(\theta, \varphi)
\end{align*}
where 
\begin{align*}
C_{L,M}(\varepsilon) = \delta_{l,L}c_{L,M} + \varepsilon c_{L,M}' + \varepsilon^2 c_{L,M}'' + O(\varepsilon^3) \qquad (\varepsilon\to 0)
\end{align*}
are complex coefficients. Note that $c_{L,M} =0$ for $L \neq l$ since the corresponding eigenspace of $\lambda_n(\mathbb{B}) =  j_{l+1/2,p}^2$ is spanned by the $(2l+1)$ eigenfunctions $u_{lm}(r,\theta,\varphi)$, $-l\leq m \leq l$. We also assume $C_{L,-M}(\varepsilon)=(-1)^{M}\overline{C_{L,M}(\varepsilon)}$ so that the eigenfunctions are always real-valued. 

As in the two-dimensional case, any eigenfunction on the perturbed domain must satisfy the boundary condition $u=0$ on $\partial \Omega_\varepsilon$ or equivalently $u(r(\theta,\varphi, \varepsilon),\theta,\varphi,\varepsilon)=0 $ for all $\theta \in[0,\pi]$ and $\varphi \in [0,2\pi)$.
Applying Taylor, we obtain
\begin{align} \label{eq:near_ball_boundary_condition}
\begin{split}
u(r(\theta,\varphi,\varepsilon),\theta,\varphi,\varepsilon) % =\sum\limits_{l=0}^{\infty} \sum\limits_{m=-l}^{l}  C_{l,m}(\varepsilon) j_l(kr(\theta,\varphi,\varepsilon)) Y_{l}^{m}(\theta, \varphi)\\
&=\sum\limits_{L,M}  C_{L,M}(\varepsilon) \bigg[ j_{L}(k_0) + j_{L}'(k_0)[(k(\varepsilon)-k_0)+ k(\varepsilon)(r(\theta,\varphi, \varepsilon)-1)] \\
&\qquad +  \frac{1}{2} j_{L}''(k_0)[(k(\varepsilon)-k_0)+ k(\varepsilon)(r(\theta,\varphi, \varepsilon)-1)]^2+O(\varepsilon^3)\bigg]Y_{L}^{M}(\theta, \varphi) \\
&= 0.
\end{split}
\end{align}
Here,
\begin{align*}
(k(\varepsilon)-k_0)+ k(\varepsilon)(r(\theta,\varphi,\varepsilon)-1) &= \varepsilon \bigg( k_1 + k_0 \sum\limits_{l',m'}  a_{l',m'}  Y_{l'}^{m'}(\theta, \varphi)\bigg) \\
&+ \varepsilon^2 \bigg( k_2 + \sum\limits_{l',m'} (k_1 a_{l',m'}  + k_0  b_{l',m'})  Y_{l'}^{m'}(\theta, \varphi) \bigg) \\
&+ O(\varepsilon^3) 
\end{align*}
as $\varepsilon \to 0$. Thus, at order $\varepsilon$, equation \eqref{eq:near_ball_boundary_condition} yields
\begin{align} \label{eq:order_eps_base}
\begin{split}
&j_{l}'(k_0) \bigg(\sum\limits_{m=-l}^{l}  c_{l,m} Y_{l}^{m}(\theta, \varphi) \bigg) \bigg( k_1 + k_0 \sum\limits_{l',m'}  a_{l',m'}  Y_{l'}^{m'}(\theta, \varphi)\bigg)\\
&\; + \sum\limits_{L,M}  c_{L,M}' j_{L}(k_0) Y_{L}^{M}(\theta, \varphi) = 0
\end{split}
\end{align}
and applying Lemma \ref{lem:sphericalharmonics_exp} now leads to
\begin{align} \label{eq:eps_1_order_equ}
\begin{split}
&j_{l}'(k_0) \bigg( k_1\sum\limits_{m=-l}^{l}  c_{l,m} Y_{l}^{m}(\theta, \varphi) + k_0 \sum\limits_{m=-l}^{l}   \sum\limits_{l',m'}   \sum\limits_{L, M}  c_{l,m}a_{l',m'}d_{l,m,l',m'}^{L,M} Y_{L}^{M}(\theta, \varphi)\bigg)\\
&\; + \sum\limits_{L, M}  c_{L,M}' j_L(k_0) Y_{L}^{M}(\theta, \varphi) = 0.
\end{split}
\end{align}
Since the spherical harmonics are linearly independent, this equation can only hold if the coefficients in front of every spherical harmonic $Y_L^M$ term are equal to zero. 

Setting coefficients of $Y_L^M$, $L = l$, $-l\leq M \leq l$, in equation \eqref{eq:eps_1_order_equ} to zero yields 
\begin{align} \label{eq:k_1_eig_val}
k_1  c_{l,M}  + k_0 \sum\limits_{m=-l}^{l}   \sum\limits_{l',m'}   c_{l,m}a_{l',m'}d_{l,m,l',m'}^{l,M} =0,
\end{align}
since $j_{l}'(k_0)\neq 0$ and $j_L(k_0)=j_{l}(k_0)=0$, so the last sum of \eqref{eq:eps_1_order_equ} does not contribute. Now set $c=(c_{l,-l},...,c_{l,l})^T \in \mathbb{C}^{2l+1}$ and the matrix $A\in \mathbb{C}^{(2l+1)\times (2l+1)}$ according to \eqref{eq:k1_matrix_setup}. It follows from the permutation symmetry and symmetry under change of sign properties of Corollary \ref{cor:sphericalharmonics_exp_coeff_properties} that $A$ is indeed Hermitian.
% \begin{align*}
% \overline{A_{m,M}} &= -k_0 \sum\limits_{l',m'}    \overline{ a_{l',m'}}d_{l,M,l',m'}^{l,m}\\
% &= -k_0\sum\limits_{l',m'}   (-1)^{m'} a_{l',-m'}d_{l,M,l',m'}^{l,m}\\
% &= -k_0\sum\limits_{l',m'}     a_{l',-m'}d_{l,-m,l',m'}^{l,-M}\\
% &= -k_0\sum\limits_{l',m'}     a_{l',m'}d_{l,-m,l',-m'}^{l,-M}\\
% &= -k_0\sum\limits_{l',m'}     a_{l',m'}d_{l,m,l',m'}^{l,M}\\
% &= A_{M,m}
% \end{align*}
 Equation \eqref{eq:k_1_eig_val} now becomes equivalent to the matrix eigenvalue problem $$Ac = k_1 c$$
 and the $(2l+1)$ real eigenvalues of $A$ sorted in increasing order become the first order corrections of $\lambda_{n'+\ell}(\Omega_\varepsilon)$, $-l \leq \ell \leq l$. We label them by $k_1^{(-l)}\leq ... \leq k_1^{(l)}$ and the corresponding orthonormalized eigenvectors by $c^{(-l)}$,..., $c^{(l)}$. We have the usual properties of the spectral decomposition. The eigenvectors $\{ c^{(\ell)} \}_{\ell=-l}^{l}$ form a basis of $\mathbb{C}^{2l+1}$ and satisfy
\begin{align*}
A = \sum\limits_{\ell=-l}^{l} k_1^{(\ell)} c^{(\ell)} (c^{(\ell)})^* \qquad &\Leftrightarrow \qquad A_{M,m} = \sum\limits_{\ell=-l}^{l} k_1^{(\ell)} c_{l,M}^{(\ell)} \overline{c_{l,m}^{(\ell)}},\\
(c^{(\ell_1)})^* c^{(\ell_2)}  = \delta_{\ell_1,\ell_2} \qquad &\Leftrightarrow \qquad \sum\limits_{m=-l}^{l}   c_{l,m}^{(\ell_1)} \overline{c_{l,m}^{(\ell_2)}}=  \delta_{\ell_1,\ell_2}.
\end{align*}
To prove equation \eqref{eq:k1_ball_trace}, we observe that
\begin{align*} % \label{eq:trace_A_is_zero}
\sum\limits_{\ell=-l}^{l} k_1^{(\ell)} = \Tr A  &= \sum\limits_{m=-l}^{l} A_{m,m} = -k_0 \sum\limits_{m=-l}^{l}\sum\limits_{l',m'}    a_{l',m'}d_{l,m,l',m'}^{l,m},
\end{align*}
where $d_{l,m,l',m'}^{l,m} = 0$ for $m' \neq 0$ due to the selection rules from Corollary \ref{cor:sphericalharmonics_exp_coeff_properties}. The summation identity from the same corollary then leads to
\begin{align*}
\sum\limits_{\ell=-l}^{l} k_1^{(\ell)}  = -k_0 \sum\limits_{l'=0}^{\infty} a_{l',0} \sum\limits_{m=-l}^{l}d_{l,m,l',0}^{l,m}= -k_0 \sqrt{\frac{1}{4 \pi}}(2l+1) a_{0,0}= 0.
\end{align*}

\end{proof}

We close the discussion of first order corrections by coming back to equation \eqref{eq:eps_1_order_equ}. Setting coefficients of $Y_L^M$, $L\neq l$, $-L\leq M \leq L$, in equation \eqref{eq:eps_1_order_equ} to zero yields
\begin{align}\label{eq:near_ball_cn1}
c_{L,M}' = - k_0 \frac{j_{l}'(k_0)}{j_L(k_0)} \sum\limits_{m=-l}^{l}   \sum\limits_{l',m'}   c_{l,m}a_{l',m'}d_{l,m,l',m'}^{L,M}, \qquad L \neq l,
\end{align}
as the first order correction for coefficients of an eigenfunction. It is natural to denote the resulting coefficients by $c_{L,M}'^{(\ell)}$ when the set of coefficients $c^{(\ell)}=(c_{l,m}^{(\ell)})_{m=-l,...,l}$ is plugged into \eqref{eq:near_ball_cn1}.

\subsection{Perturbation expansions for eigenfrequencies of the ball: second order}

\begin{proposition} \label{prop:ball_second_order_sumk2}
Under the assumptions of Proposition \ref{prop:ball_first_order}, let $k_2^{(\ell)}$ denote the second order perturbation expansion coefficient of $k(\varepsilon)=\sqrt{\lambda_{n'+\ell}(\Omega_\varepsilon)}$. Then,
\begin{align}  \label{eq:k_2_ball_sum}
\sum\limits_{\ell=-l}^{l}\frac{k_2^{(\ell)}}{k_0 } = (2l+1) \sum\limits_{L,M} \bigg(\frac{1}{2 \pi}  +  \sum\limits_{l'\neq l} k_0\frac{j_{l'}'(k_0)}{j_{l'}(k_0)} \frac{(2l'+1)}{4\pi (2L+1)}  \left(\langle l, 0, l',0|L,0\rangle \right)^2  \bigg) |a_{L,M}|^2
\end{align}
\end{proposition}

\begin{proof}
At order $\varepsilon^2$, equation \eqref{eq:near_ball_boundary_condition} leads to 
\begin{align} \label{eq:second_order_base}
\begin{split}
& \sum\limits_{m=-l}^{l}  c_{l,m}Y_{l}^{m}(\theta, \varphi) \Bigg[j_{l}'(k_0)\bigg( k_2 + \sum\limits_{l', m'}  (k_1 a_{l',m'} + k_0 b_{l',m'}) Y_{l'}^{m'}(\theta, \varphi) \bigg)\\
&\hspace{4cm}+ \frac{1}{2}j_{l}''(k_0)\bigg( k_1 + k_0 \sum\limits_{l',m'}  a_{l',m'}  Y_{l'}^{m'}(\theta, \varphi)\bigg)^2 \Bigg] \\
&\quad + \sum\limits_{L,M} c_{L,M}' j_L'(k_0)Y_{L}^{M}(\theta, \varphi) \bigg( k_1 + k_0 \sum\limits_{l',m'} a_{l',m'}  Y_{l'}^{m'}(\theta, \varphi)\bigg)\\
&\quad+ \sum\limits_{L,M}  c_{L,M}'' j_L(k_0) Y_{L}^{M}(\theta, \varphi) = 0.
\end{split}
\end{align}
To extract $k_2$ from this equation, it is necessary to set terms in front of $Y_L^M$, $L = l$, $-l\leq M \leq l$, to zero. However, we need to rewrite all appearing products of spherical harmonics in the basis of spherical harmonics. We begin by rewriting the first summand with the square brackets in equation \eqref{eq:second_order_base}. First, expanding the square gives
\begin{align*}
&\bigg( k_1 + k_0 \sum\limits_{l', m'}  a_{l',m'}  Y_{l'}^{m'}(\theta, \varphi)\bigg)^2 \\
% = & \; k_1^2+ 2k_0 k_1 \sum\limits_{l', m'}  a_{l',m'}  Y_{l'}^{m'}(\theta, \varphi)+k_0^2 \sum\limits_{l_1, m_1} \sum\limits_{l_2, m_2}  a_{l_1,m_1}  a_{l_2,m_2} Y_{l_1}^{m_1}(\theta, \varphi)  Y_{l_2}^{m_2}(\theta, \varphi)\\
= & \; k_1^2+  \sum\limits_{l', m'} \bigg(2k_0 k_1a_{l',m'}  +k_0^2   \sum\limits_{l_1, m_1} \sum\limits_{l_2, m_2}  a_{l_1,m_1}  a_{l_2,m_2} d_{l_1,m_1,l_2,m_2}^{l',m'}\bigg)  Y_{l'}^{m'}(\theta, \varphi).
\end{align*}
Thus, the expression in square brackets from \eqref{eq:second_order_base} becomes
\begin{align} \label{eq:second_order_squarebrackets}
\Bigg[ \quad ... \quad \Bigg] = \, \bigg( j_{l}'(k_0)k_2 + \frac{1}{2}j_{l}''(k_0)k_1^2 \bigg) +\sum\limits_{l',m'} e_{l',m'} Y_{l'}^{m'}(\theta, \varphi)
% = \, \Bigg( j_{l^*}'(k_0)k_2 + \frac{1}{2}j_{l^*}''(k_0)k_1^2 \Bigg) +\sum\limits_{l'=0}^{\infty}\sum\limits_{m'=-l'}^{l'} \Bigg( j_{l^*}'(k_0)k_1 a_{l',m'} + j_{l^*}'(k_0)k_0 b_{l',m'}\\
% & +j_{l^*}''(k_0)k_0 k_1a_{l',m'} + \frac{1}{2}j_{l^*}''(k_0)k_0^2   \sum\limits_{l_1=0}^{\infty} \sum\limits_{m_1=-l_1}^{l_1} \sum\limits_{l_2=0}^{\infty} \sum\limits_{m_2=-l_2}^{l_2}  a_{l_1,m_1}  a_{l_2,m_2} d_{l_1,m_1,l_2,m_2}^{l',m'} \Bigg)  Y_{l'}^{m'}(\theta, \varphi)\\
\end{align}
where
\begin{align*}
    e_{l',m'} &= j_{l}'(k_0)k_1 a_{l',m'} + j_{l}'(k_0)k_0 b_{l',m'} +j_{l}''(k_0)k_0 k_1a_{l',m'} \\
    & \;\quad + \frac{1}{2}j_{l}''(k_0)k_0^2   \sum\limits_{l_1, m_1}\sum\limits_{l_2,m_2}  a_{l_1,m_1}  a_{l_2,m_2} d_{l_1,m_1,l_2,m_2}^{l',m'} 
\end{align*}
We insert \eqref{eq:second_order_squarebrackets} into equation \eqref{eq:second_order_base}, expand again products of spherical harmonics and set coefficients of $Y_L^M$, $L = l$, $-l\leq M \leq l$, to zero. Since $j_{l}(k_0)=0$, all terms with coefficients $c_{l,m}''$ vanish. It follows
\begin{align} \label{eq:second_order_base2}
\begin{split}
& c_{l,M} \left( j_{l}'(k_0)k_2 + \frac{1}{2}j_{l}''(k_0)k_1^2 \right)  +\sum\limits_{m=-l}^{l}  \sum\limits_{l',m'}  c_{l,m} e_{l',m'}d_{l,m,l',m'}^{l,M}\\
& \quad +c_{l,M}' j_{l}'(k_0)k_1 +      k_0 \sum\limits_{l'',m''} \sum\limits_{l',m'} c_{l'',m''}' j_{l''}'(k_0) a_{l',m'}  d_{l'',m'',l',m'}^{l,M} =0.
\end{split} 
\end{align}
Note that all terms that involve $k_1$, $k_2$, in particular $e_{l',m'}$, and the coefficients $c_{l,m},c_{l,m}'$ are implicitly equipped with an additional $^{(\ell)}$ superscript. 

We multiply equation \eqref{eq:second_order_base2} by $\overline{c_{l,M}}/(k_0 j_{l}'(k_0))$ and then sum over $-l\leq M \leq l$ as well as over all resulting equations for all $-l \leq  \ell \leq l$ indexing the admissible eigenvectors $c^{(\ell)}$ of $A$ from Proposition \ref{prop:ball_first_order}. This yields
\begin{align}  \label{eq:second_order_base_summed}
\begin{split}
& \sum\limits_{\ell=-l}^{l} \sum\limits_{M=-l}^{l} \bigg[ |c_{l,M}^{(\ell)}|^2 \bigg( \frac{k_2^{(\ell)}}{k_0} + \frac{j_{l}''(k_0)}{2k_0j_{l}'(k_0)}(k_1^{(\ell)})^2 \bigg)  \\
&\quad +\frac{1}{k_0j_{l}'(k_0)}\sum\limits_{m=-l}^{l}  \sum\limits_{l',m'} c_{l,m}^{(\ell)}\overline{c_{l,M}^{(\ell)}} e_{l',m'}^{(\ell)}d_{l,m,l',m'}^{l,M}\\
&\quad  +c_{l,M}'^{(\ell)}\overline{c_{l,M}^{(\ell)}} \frac{k_1^{(\ell)}}{k_0} +       \sum\limits_{l'',m''} \sum\limits_{l',m'} c_{l'',m''}'^{(\ell)}\overline{c_{l,M}^{(\ell)}} \frac{j_{l''}'(k_0)}{j_{l}'(k_0)} a_{l',m'}  d_{l'',m'',l',m'}^{l,M} \bigg]=0. 
\end{split}
\end{align}
First of all, we note that the first line in \eqref{eq:second_order_base_summed} evaluates to
\begin{align*} 
\sum\limits_{\ell=-l}^{l} \sum\limits_{M=-l}^{l}  |c_{l,M}^{(\ell)}|^2 \bigg( \frac{k_2^{(\ell)}}{k_0} + \frac{j_{l}''(k_0)}{2k_0j_{l}'(k_0)}(k_1^{(\ell)})^2 \bigg) = \sum\limits_{\ell=-l}^{l} \bigg( \frac{k_2^{(\ell)}}{k_0} - \frac{(k_1^{(\ell)})^2 }{k_0^2} \bigg),
\end{align*}
since $ \sum_{M=-l}^{l}  |c_{l,M}^{(\ell)}|^2 = \Vert c^{(\ell)} \Vert^2 =1$ for any $-l \leq \ell \leq l$ and $ j_{l}''(k_0) = - 2j_{l}'(k_0)/k_0$
by the spherical Bessel equation. Now we take a closer look at the last line of \eqref{eq:second_order_base_summed}. We split the sum $\sum_{l''}$ from the last term into $\sum_{l'' = l}+ \sum_{l'' \neq l}$. When $l''=l$ we recognize the matrix elements from the first order discussion and get
% \begin{align*}
% \sum\limits_{m''=-l}^{l} \sum\limits_{l', m'} c_{l,m''}'^{(\ell)} \frac{j_{l}'(k_0)}{j_{l^*}'(k_0)} a_{l',m'}  d_{l,m'',l',m'}^{l,M} =- \frac{1}{k_0} \sum\limits_{m''=-l}^{l}  c_{l,m''}'^{(\ell)}   A_{M,m''} .
%  \end{align*} 
%  Let us now perform the full summation for $l=l^*$. Summing over $M$ and $j$ gives
 \begin{align*}
&\sum\limits_{\ell=-l}^{l} \sum\limits_{M=-l}^{l}  \sum\limits_{m''=-l}^{l} \sum\limits_{l',m'} c_{l,m''}'^{(\ell)}\overline{c_{l,M}^{(\ell)}} \frac{j_{l}'(k_0)}{j_{l}'(k_0)} a_{l',m'}  d_{l,m'',l',m'}^{l,M}  \\
& \qquad=- \frac{1}{k_0} \sum\limits_{\ell=-l}^{l}  \sum\limits_{m''=-l}^{l} \sum\limits_{M=-l}^{l}  c_{l,m''}'^{(\ell)}  \overline{c_{l,M}^{(\ell)}} A_{M,m''}=-\sum\limits_{\ell=-l}^{l}  \sum\limits_{m''=-l}^{l} c_{l,m''}'^{(\ell)}  \overline{c_{l,m''}^{(\ell)}}  \frac{k_1^{(\ell)}}{k_0}  
 \end{align*}  
since the matrix $A$ was diagonalized with the eigenvectors $c^{(\ell)}$. Thus, $\sum_{l'' = l}$ cancels exactly the former terms from the last line of \eqref{eq:second_order_base_summed} and we obtain
\begin{align} \label{eq:second_order_base_summed_step2}
\sum\limits_{\ell=-l}^{l}  \frac{k_2^{(\ell)}}{k_0} &= \sum\limits_{\ell=-l}^{l}  \frac{(k_1^{(\ell)})^2 }{k_0^2} + S_1 + S_2 
\end{align}
where
\begin{align} 
S_1 &= -\frac{1}{k_0j_{l}'(k_0)}\sum\limits_{\ell=-l}^{l} \sum\limits_{M=-l}^{l} \sum\limits_{m=-l}^{l}  \sum\limits_{l',m'}  c_{l,m}^{(\ell)}\overline{c_{l,M}^{(\ell)}} e_{l',m'}^{(\ell)} d_{l,m,l',m'}^{l,M} , \label{eq:S_1_def} \\
S_2 &= -     \sum\limits_{\ell=-l}^{l} \sum\limits_{M=-l}^{l}  \sum\limits_{l''\neq l, m''}  \sum\limits_{l',m'} c_{l'',m''}'^{(\ell)}\overline{c_{l,M}^{(\ell)}} \frac{j_{l''}'(k_0)}{j_{l}'(k_0)} a_{l',m'}  d_{l'',m'',l',m'}^{l,M} .  \label{eq:S_2_def}
\end{align}
The proposition then follows from the next two lemmas which evaluate $S_1$ and $S_2$.
\end{proof}

\begin{lemma} \label{lem:lemma_S_1_end_result}
    \begin{align} \label{eq:lemma_S_1_end_result}
        S_1 =(2l+1) \frac{1}{2 \pi}  \sum\limits_{L,M} |a_{L,M}|^2 -  \sum\limits_{\ell=-l}^{l}\frac{(k_1^{(\ell)})^2}{k_0^2 }  
    \end{align}
\end{lemma}

\begin{lemma} \label{lem:lemma_S_2_end_result}
    \begin{align} \label{eq:lemma_S_2_end_result}
        S_2 = \sum\limits_{L,M}  \bigg(  \sum\limits_{l'\neq l} k_0\frac{j_{l'}'(k_0)}{j_{l'}(k_0)} \frac{(2l+1)(2l'+1)}{4\pi (2L+1)}  \left(\langle l, 0, l',0|L,0\rangle \right)^2 \bigg) |a_{L,M}|^2.
    \end{align}
\end{lemma}

\begin{proof}[Proof of Lemma \ref{lem:lemma_S_1_end_result}]
We begin by inserting $e_{l',m'}$ back into \eqref{eq:S_1_def} and separating the result into three sums. With $ j_{l}''(k_0) = - 2j_{l}'(k_0)/k_0$, we get
\begin{align*}  
% &\frac{1}{k_0j_{l^*}'(k_0)}\sum\limits_{j=-l^*}^{l^*} \sum\limits_{M=-l^*}^{l^*} \sum\limits_{m=-l^*}^{l^*}  \sum\limits_{l'=0}^{\infty}\sum\limits_{m'=-l'}^{l'}  c_{l^*,m}^{(j)}\overline{c_{l^*,M}^{(j)}} e_{l',m'}^{(j)} d_{l^*,m,l',m'}^{l^*,M} \\
S_1&= -\frac{1}{k_0j_{l}'(k_0)}\sum\limits_{\ell=-l}^{l} \sum\limits_{M=-l}^{l} \sum\limits_{m=-l}^{l}  \sum\limits_{l',m'} c_{l,m}^{(\ell)} \overline{c_{l,M}^{(\ell)}} \Bigg( j_{l}'(k_0)k_1^{(\ell)} a_{l',m'} + j_{l}'(k_0)k_0 b_{l',m'}\\
& \qquad +j_{l}''(k_0)k_0 k_1^{(\ell)} a_{l',m'} + \frac{1}{2}j_{l}''(k_0)k_0^2   \sum\limits_{l_1,m_1} \sum\limits_{l_2,m_2}  a_{l_1,m_1}  a_{l_2,m_2} d_{l_1,m_1,l_2,m_2}^{l',m'} \Bigg) d_{l,m,l',m'}^{l,M}\\
% &=  - \sum\limits_{j,M,m,l',m'}  c_{l^*,m}^{(j)} \overline{c_{l^*,M}^{(j)}}\Bigg( \frac{k_1^{(j)}}{k_0} a_{l',m'} + b_{l',m'} - 2\frac{k_1^{(j)}}{k_0} a_{l',m'} \\
% &\qquad -\sum\limits_{l_1=0}^{\infty} \sum\limits_{m_1=-l_1}^{l_1} \sum\limits_{l_2=0}^{\infty} \sum\limits_{m_2=-l_2}^{l_2}  a_{l_1,m_1}  a_{l_2,m_2} d_{l_1,m_1,l_2,m_2}^{l',m'} \Bigg) d_{l^*,m,l',m'}^{l^*,M}\\
&= S_1^{(1)} + S_1^{(2)} + S_1^{(3)}
\end{align*}
where
\begin{align*}
S_1^{(1)}&= \sum\limits_{\ell,M,m,l',m'}   c_{l,m}^{(\ell)} \overline{c_{l,M}^{(\ell)}} \frac{k_1^{(\ell)}}{k_0} a_{l',m'}d_{l,m,l',m'}^{l,M} , \\
S_1^{(2)} &= - \sum\limits_{\ell,M,m,l',m'} c_{l,m}^{(\ell)} \overline{c_{l,M}^{(\ell)}}  b_{l',m'} d_{l,m,l',m'}^{l,M} ,\\
S_1^{(3)} &= \sum\limits_{\ell,M,m,l',m'} \sum\limits_{l_1,m_1,l_2,m_2} c_{l,m}^{(\ell)} \overline{c_{l,M}^{(\ell)}} a_{l_1,m_1}  a_{l_2,m_2} d_{l_1,m_1,l_2,m_2}^{l',m'} d_{l,m,l',m'}^{l,M}.
\end{align*}
We treat each of the three sums separately. The first sum leads to 
\begin{align} \label{eq:S_11_result}
\begin{split}
  S_1^{(1)} &=\frac{1}{k_0}\sum\limits_{\ell=-l}^{l}\frac{k_1^{(\ell)}}{k_0} \sum\limits_{M=-l}^{l} \sum\limits_{m=-l}^{l}    c_{l,m}^{(\ell)} \overline{c_{l,M}^{(\ell)}}   \left(k_0\sum\limits_{l',m'} a_{l',m'}d_{l,m,l',m'}^{l,M} \right) \\
  &=-\frac{1}{k_0}\sum\limits_{\ell=-l}^{l}\frac{k_1^{(\ell)}}{k_0} \sum\limits_{M=-l}^{l} \sum\limits_{m=-l}^{l}    c_{l,m}^{(\ell)} \overline{c_{l,M}^{(\ell)}}  A_{M,m} =-\sum\limits_{\ell=-l}^{l}\frac{(k_1^{(\ell)})^2}{k_0^2 }.  
\end{split}
\end{align}
For the second sum, we interpret
\begin{align*}
    D_{M,m} =  \sum\limits_{l',m'}  b_{l',m'} d_{l,m,l',m'}^{l,M} 
\end{align*}
as entries of a matrix $D \in \mathbb{C}^{(2l+1)\times (2l+1)}$. Then, 
\begin{align*}
    S_1^{(2)} = - \Tr D = -  \sum\limits_{M=-l}^{l} \sum\limits_{l',m'}   b_{l',m'} d_{l,M,l',m'}^{l,M} .
\end{align*}
Selection rules of Corollary \ref{cor:sphericalharmonics_exp_coeff_properties} imply that only terms with $m'=0$ survive. Using the summation identity (7) of Corollary \ref{cor:sphericalharmonics_exp_coeff_properties} and \eqref{eq:a00_b00} eventually gives
\begin{align} \label{eq:S_12_result}
\begin{split}
    S_1^{(2)}&=-\sum\limits_{l'} b_{l',0} \sum\limits_{M=-l}^{l}  d_{l,M,l',0}^{l,M} \\
&= -(2l+1) \sqrt{\frac{1}{4\pi}} b_{0,0}  \\
&= (2l+1) \frac{1}{4\pi} \sum\limits_{L=0}^{\infty} \sum\limits_{M=-L}^{L}  |a_{L,M}|^2.
\end{split}
\end{align}
The third sum can be rearranged to
\begin{align*}
S_1^{(3)}=\sum\limits_{l',m'}\left(\sum\limits_{l_1, m_1} \sum\limits_{l_2,m_2} a_{l_1,m_1}  a_{l_2,m_2} d_{l_1,m_1,l_2,m_2}^{l',m'} \right) \left( \sum\limits_{\ell=-l}^{l} \sum\limits_{M=-l}^{l} \sum\limits_{m=-l}^{l}  c_{l,m}^{(\ell)} \overline{c_{l,M}^{(\ell)}}d_{l,m,l',m'}^{l,M} \right)
\end{align*}
We focus on the latter factor that consists of a triple sum. We define matrices $B^{l',m'} \in  \mathbb{C}^{(2l+1) \times (2l+1)}$ by entries
\begin{align*}
B_{M,m}^{l',m'} = d_{l,m,l',m'}^{l,M}.
\end{align*} 
Then, by selection rules and summation identity (7) of Corollary \ref{cor:sphericalharmonics_exp_coeff_properties},
\begin{align*}
\sum\limits_{\ell=-l}^{l} \sum\limits_{M=-l}^{l} \sum\limits_{m=-l}^{l}  c_{l,m}^{(\ell)} \overline{c_{l,M}^{(\ell)}}d_{l,m,l',m'}^{l,M}  = \Tr B^{l',m'} &= \sum\limits_{M=-l}^{l} B_{M,M}^{l',m'} \\
&=\sum\limits_{M=-l}^{l} d_{l,M,l',m'}^{l,M} \\
&=\delta_{m',0}\sum\limits_{M=-l}^{l} d_{l,M,l',0}^{l,M} \\
&= \sqrt{\frac{1}{4 \pi}}(2l+1)\delta_{m',0}\delta_{l',0} .
\end{align*}
Therefore,
\begin{align} \label{eq:S_13_result}
\begin{split}
S_1^{(3)} &=(2l+1)\sqrt{\frac{1}{4 \pi}}\sum\limits_{l_1,m_1} \sum\limits_{l_2,m_2} a_{l_1,m_1}  a_{l_2,m_2} d_{l_1,m_1,l_2,m_2}^{0,0} \\ 
&=(2l+1)\sqrt{\frac{1}{4 \pi}}\sum\limits_{l_1,m_1} \sum\limits_{l_2,m_2} a_{l_1,m_1}  a_{l_2,m_2} (-1)^{m_2} d_{0,0,l_2,m_2}^{l_1,-m_1} \\ 
&=(2l+1)\sqrt{\frac{1}{4 \pi}}\sum\limits_{l_1,m_1} \sum\limits_{l_2,m_2} a_{l_1,m_1}  a_{l_2,-m_2} (-1)^{m_2} d_{0,0,l_2,-m_2}^{l_1,-m_1} \\
&=(2l+1)\sqrt{\frac{1}{4 \pi}}\sum\limits_{l_1,m_1} \sum\limits_{l_2,m_2} a_{l_1,m_1}  \overline{a_{l_2,m_2}}  d_{0,0,l_2,m_2}^{l_1,m_1} \\
&=(2l+1)\sqrt{\frac{1}{4 \pi}}\sum\limits_{l_1,m_1} \sum\limits_{l_2,m_2} a_{l_1,m_1}  \overline{a_{l_2,m_2}} \sqrt{\frac{1}{4 \pi}} \delta_{l_1,l_2}\delta_{m_1,m_2}  \\
&=(2l+1)\frac{1}{4 \pi}\sum\limits_{L,M}  |a_{L,M}|^2. 
\end{split}
\end{align}
Putting together \eqref{eq:S_11_result}, \eqref{eq:S_12_result} and \eqref{eq:S_13_result} proves the lemma.
\end{proof}

\begin{proof}[Proof of Lemma \ref{lem:lemma_S_2_end_result}]
Since $l''$ in \eqref{eq:S_2_def} runs over $l''\neq l$, we can use equation \eqref{eq:near_ball_cn1} for $c_{l'',m''}'^{(\ell)}$. This gives us
\begin{align} \label{eq:k2_intermediate_step}
 S_2=   \sum\limits_{\ell,m_1,M} c_{l,m_1}^{(\ell)}\overline{c_{l,M}^{(\ell)}}K_{M,m_1}
\end{align}
where we define the matrix $K \in \mathbb{C}^{(2l+1) \times (2l+1)}$ by
\begin{align*}
% K&=(K_{M,m_1})_{M,m_1=-l^*,...,l^*} \in \mathbb{C}^{(2l^*+1) \times (2l^*+1)}, \\
K_{M,m_1} &= \sum\limits_{\substack{l''\neq l, m'',\\ l',m',l_2,m_2}}   k_0\frac{j_{l''}'(k_0)}{j_{l''}(k_0)} a_{l_2,m_2}a_{l',m'} d_{l,m_1,l_2,m_2}^{l'',m''}   d_{l'',m'',l',m'}^{l,M} .
\end{align*} 
Since $\{ c^{(\ell)} \}_{\ell=-l,...,l}$ forms an orthonormal basis of $\mathbb{C}^{2l+1}$, we have
\begin{align*}
S_2 = \Tr K &= \sum\limits_{M=-l}^{l} K_{M,M}\\
% &= \sum\limits_{M=-l}^{l} \sum\limits_{\substack{l''\neq l, m'',\\ l',m',l_2,m_2}}   k_0\frac{j_{l''}'(k_0)}{j_{l''}(k_0)} a_{l_2,m_2}a_{l',m'} d_{l,M,l_2,m_2}^{l'',m''}   d_{l'',m'',l',m'}^{l,M}\\
&=\sum\limits_{l_2,m_2,l',m'}      a_{l_2,m_2}a_{l',m'} \sum\limits_{l''\neq l} k_0\frac{j_{l''}'(k_0)}{j_{l''}(k_0)} \sum\limits_{m'',M}  d_{l,M,l_2,m_2}^{l'',m''}   d_{l'',m'',l',m'}^{l,M}.\\
\end{align*}
To evaluate this sum, we use selection rules of Corollary \ref{cor:sphericalharmonics_exp_coeff_properties} to see that only terms with $M+m_2=m''=M-m'$ survive. Together with permutation symmetry and orthogonality relations this yields
\begin{align*}
\sum\limits_{m'',M}  d_{l,M,l_2,m_2}^{l'',m''}   d_{l'',m'',l',m'}^{l,M} &= \sum\limits_{m'',M} (-1)^{-M}(-1)^{-m''} d_{l,M,l'',-m''}^{l_2,-m_2}   d_{l'',m'',l,-M}^{l',-m'} \\
&= \sum\limits_{m^*,M} (-1)^{m''-M} d_{l,M,l'',m^*}^{l_2,-m_2}   d_{l'',-m^*,l,-M}^{l',-m'} \\
&= (-1)^{m'} \sum\limits_{m^*,M}  d_{l,M,l'',m^*}^{l_2,-m_2}   d_{l'',m^*,l,M}^{l',m'} \\
&= (-1)^{m'} \sum\limits_{m^*,M}  d_{l,M,l'',m^*}^{l_2,-m_2}   d_{l,M,l'',m^*}^{l',m'} \\
&= (-1)^{m'} \frac{(2l+1)(2l''+1)}{4\pi (2l_2+1)}  \left(\langle l, 0, l'',0|l_2,0\rangle \right)^2 \delta_{l_2,l'} \delta_{-m_2,m'} ,
\end{align*}
where in the second line, we have just relabeled $m^*=-m''$. Inserting this result and relabeling $l''$ as $l'$ above establishes \eqref{eq:lemma_S_2_end_result} and finishes the proof of the lemma.
\end{proof}

Another ingredient that we need in the next section is an expression for sums of squares of first order perturbation expansion coefficients.

\begin{proposition} \label{prop:ball_second_order_sumk1_sqr}
Under the assumptions of Proposition \ref{prop:ball_first_order}, it holds
\begin{align}  \label{eq:k_1_ball_squared_sum}
\sum\limits_{\ell=-l}^{l}\frac{(k_1^{(\ell)})^2}{k_0^2 } = (2l+1)\sum\limits_{L,M}   \frac{(2l+1)}{4\pi (2L+1)}  \left(\langle l, 0, l,0| L,0\rangle \right)^2 |a_{L,M}|^2.
\end{align}
\end{proposition}

\begin{proof}
The sum in question is nothing but $(1/k_0^2)\, \Tr A^2$ and therefore
\begin{align*}
\sum\limits_{\ell=-l}^{l}\frac{(k_1^{(\ell)})^2}{k_0^2 }&= \frac{1}{k_0^2 }\sum\limits_{m=-l}^{l} \sum\limits_{M=-l}^{l} \bigg( -k_0\sum\limits_{l_1,m_1}    a_{l_1,m_1}d_{l,M,l_1,m_1}^{l,m}\bigg) \bigg( -k_0\sum\limits_{l_2,m_2}    a_{l_2,m_2}d_{l,m,l_2,m_2}^{l,M}\bigg)\\
&=  \sum\limits_{l_1,m_1} \sum\limits_{l_2,m_2} a_{l_1,m_1}   a_{l_2,m_2} \bigg(  \sum\limits_{m=-l}^{l} \sum\limits_{M=-l}^{l} d_{l,M,l_1,m_1}^{l,m} d_{l,m,l_2,m_2}^{l,M}\bigg) .
\end{align*}
Here, by selection rules ($M+m_1 = m = M-m_2$), permutation symmetry, symmetry under change of sign and orthogonality,
\begin{align*}
\sum\limits_{m=-l}^{l} \sum\limits_{M=-l}^{l} d_{l,M,l_1,m_1}^{l,m} d_{l,m,l_2,m_2}^{l,M} &= \sum\limits_{m=-l}^{l} \sum\limits_{M=-l}^{l} (-1)^M  (-1)^m d_{l,M,l,-m}^{l_1,-m_1} d_{l,m,l,-M}^{l_2,-m_2} \\
&= \sum\limits_{m=-l}^{l} \sum\limits_{M=-l}^{l} (-1)^{M-m} d_{l,M,l,m}^{l_1,-m_1} d_{l,-m,l,-M}^{l_2,-m_2} \\
&= (-1)^{m_2} \sum\limits_{m=-l}^{l} \sum\limits_{M=-l}^{l}  d_{l,M,l,m}^{l_1,-m_1} d_{l,m,l,M}^{l_2,m_2} \\
&= (-1)^{m_2} \frac{(2l+1)(2l+1)}{4\pi (2l_1+1)}  \left(\langle l, 0, l,0|l_1,0\rangle \right)^2 \delta_{l_1,l_2} \delta_{-m_1,m_2}, 
\end{align*}
hence, after relabeling $l_1,m_1$ as $L,M$ we arrive at \eqref{eq:k_1_ball_squared_sum}.
\end{proof}

\subsection{Perturbation expansions for eigenvalue sums of the ball}

In this section, we put together the pieces of the last section and show perturbation expansions for sums of renormalized eigenvalues. We begin with the case where $\lambda_{n}(\mathbb{B}) < \lambda_{n+1}(\mathbb{B}) $.

\begin{proposition} \label{prop:prop3d_Sn_exp_quadratic}
Let $\lambda_n(\mathbb{B})=j_{l+1/2,p}^2$. If the $n$-th eigenvalue of the unit ball is a single eigenvalue or the last instance of a degenerate eigenvalue, i.e.~$\lambda_{n}(\mathbb{B}) < \lambda_{n+1}(\mathbb{B}) $, then
    \begin{align} \label{eq:prop3d_Sn_exp_case2}
    \begin{split}
S_n(\Omega_\varepsilon)& =  S_n(\mathbb{B}) \\
&\; + \varepsilon^2 \cdot \Bigg[ |\mathbb{B}|^{2/3} \sum_{L,M}\sum\limits_{j_{l+1/2,p}^2 \leq \lambda_n(\mathbb{B})} (2l+1) j_{l+1/2,p}^2 \bigg(\frac{1}{ \pi}  + \frac{(2l+1)}{4\pi (2L+1)}  \left(\langle l, 0, l,0|L,0\rangle \right)^2 \\
&\;\qquad \qquad+ 2\sum\limits_{l'\neq l} j_{l+1/2,p}\frac{j_{l'}'(j_{l+1/2,p})}{j_{l'}(j_{l+1/2,p})} \frac{(2l'+1)}{4\pi (2L+1)}  \left(\langle l, 0, l',0|L,0\rangle \right)^2  \bigg) |a_{L,M}|^2 \Bigg] \\
&\;+ O(\varepsilon^3)  \qquad\qquad (\varepsilon\to 0).
\end{split}
\end{align}
\end{proposition}

\begin{proof}
Every renormalized eigenvalue satisfies
\begin{align*} 
|\Omega_\varepsilon|^{2/3}\lambda_k(\Omega_\varepsilon) = |\mathbb{B}|^{2/3}k_0^2 \bigg( 1+ 2\varepsilon\frac{k_1}{k_0} + \varepsilon^2 \left( 2\frac{k_2}{k_0} + \frac{k_1^2}{k_0^2}\right)  + O(\varepsilon^3) \bigg) \qquad (\varepsilon\to 0),
\end{align*}
with $k_1$, $k_2$ depending on $k$ and having the properties established in the previous section. To find a perturbation expansion for $S_n(\Omega_\varepsilon)$, we split $S_n(\Omega_\varepsilon)$ according to eigenvalue clusters of the ball. We start by separating the top cluster from all lower ones, so with the notation of Proposition \ref{prop:ball_first_order}, we have
\begin{align} \label{eq:eigval_sum_split_up3d}
\begin{split}
S_n(\Omega_\varepsilon) &= |\Omega_\varepsilon|^{2/3}\bigg(\sum\limits_{\lambda_k(\mathbb{B}) < \lambda_n(\mathbb{B}) }   \lambda_k(\Omega_\varepsilon) +   \sum\limits_{\ell=-l}^{j}\lambda_{n'+\ell}(\Omega_\varepsilon) \bigg) \\
&= S_{n'-l-1}(\Omega_\varepsilon) +   |\Omega_\varepsilon|^{2/3}\sum\limits_{\ell=-l}^{j}\lambda_{n'+\ell}(\Omega_\varepsilon) ,
\end{split}
\end{align}
where
\begin{align} \label{eq:eigval_sum_split_up3d_topcluster}
\begin{split}
\sum\limits_{\ell=-l}^{j}\lambda_{n'+\ell}(\Omega_\varepsilon)  &= (j+l+1) k_0^2 + \varepsilon \cdot 2k_0^2\sum\limits_{\ell=-l}^{j}\frac{k_1^{(\ell)}}{k_0} \\
& \qquad+ \varepsilon^2 \cdot  k_0^2 \bigg(2\sum\limits_{\ell=-l}^{j}\frac{k_2^{(\ell)}}{k_0} + \sum\limits_{\ell=-l}^{j}\frac{(k_1^{(\ell)})^2}{k_0^2}\bigg)  + O(\varepsilon^3) 
% \sum\limits_{\ell=-l}^{l}k_0^2 \bigg( 1+ 2\varepsilon\frac{k_1^{(\ell)}}{k_0} + \varepsilon^2 \bigg( 2\frac{k_2^{(\ell)}}{k_0} + \frac{(k_1^{(\ell)})^2}{k_0^2}\bigg)  + O(\varepsilon^3) \bigg)\\
\end{split}
\end{align}
as $\varepsilon\searrow 0$ where $k_0 = j_{l+1/2,p}$ and $k_1^{(\ell)}$, $k_2^{(\ell)}$ are the correction terms for this particular $k_0$ under a given perturbation. The assumption that $\lambda_n(\mathbb{B})$ is a single eigenvalue or the last instance of a degenerate eigenvalue means that $j=l$. One can therefore apply Propositions \ref{prop:ball_first_order},  \ref{prop:ball_second_order_sumk2} and \ref{prop:ball_second_order_sumk1_sqr} to see that first order corrections vanish and only second order terms remain. For the sum $S_{n'-l-1}(\Omega_\varepsilon)$, one observes that $\lambda_{n'-l-1}(\mathbb{B})$ is again a single eigenvalue or the last instance of a degenerate eigenvalue, so the same reasoning can be applied for the next lower cluster. Summing all terms for all eigenvalue clusters and using $|\Omega_\varepsilon|^{2/3}=|\mathbb{B}|^{2/3} + O(\varepsilon^3)$ yields \eqref{eq:prop3d_Sn_exp_case2}.
\end{proof}

\begin{proposition} \label{prop:prop3d_Sn_exp_linear}
Let $\lambda_n(\mathbb{B})=j_{l+1/2,p}^2$ with $l\geq 1$ be a degenerate eigenvalue. If the $n$-th eigenvalue of the unit ball is not the last instance of that degenerate eigenvalue, i.e.~$\lambda_n(\mathbb{B})=\lambda_{n+1}(\mathbb{B})$, then 
    \begin{align} \label{eq:prop3d_Sn_exp_case1}
        S_n(\Omega_\varepsilon) = S_n(\mathbb{B}) + \varepsilon \cdot 2 |\mathbb{B}|^{2/3}  j_{l+1/2,p}^2 \sum_{\ell = -l}^{j} \frac{k_1^{(\ell)}}{k_0}+O(\varepsilon^2) \qquad (\varepsilon\searrow 0)
    \end{align}
and there exists a perturbation with 
\begin{align*}
    \sum_{\ell = -l}^{j} \frac{k_1^{(\ell)}}{k_0} < 0 .
\end{align*}
\end{proposition}

\begin{proof}
Equation \eqref{eq:prop3d_Sn_exp_case1} follows immediately from \eqref{eq:eigval_sum_split_up3d}, \eqref{eq:eigval_sum_split_up3d_topcluster} and $|\Omega_\varepsilon|^{2/3}=|\mathbb{B}|^{2/3} + O(\varepsilon^3)$. The assumption that the $n$-th eigenvalue of the unit ball is not the last instance of a degenerate eigenvalue means that $j<l$. To prove the existence part of the proposition, it is sufficient to find a perturbation with $k_1^{(\ell)} \neq 0 $ for some $-l \leq \ell \leq l$, since $\sum_{\ell=-l}^l k_1^{(\ell)} = 0$ and $\{ k_1^{(\ell)} \}_{\ell=-l}^{l}$ being sorted in increasing order then automatically implies $\sum_{\ell=-l}^j k_1^{(\ell)} < 0$ for any $j<l$. 

We consider the perturbation defined by
\begin{align*}
a_{L,M} = \begin{cases}
1 & \text{ if } (L,M)=(2,0),\\
0 & \text{ else.}
\end{cases}
\end{align*}
The entries of the matrix $A$ from Proposition \ref{prop:ball_first_order} are then given by $A_{M,m} =-k_0  d_{l,m,2,0}^{l,M}$. Selection rules imply that $d_{l,m,2,0}^{l,M}=0$ if $m\neq M$, so $A$ is diagonal. The first order corrections $k_1^{(\ell)}$ are thus just the diagonal entries of $A$ sorted by increasing value. Permutation symmetry yields $A_{m,m} = -k_0(-1)^{-m}d_{l,m,l,-m}^{2,0}$. We claim that $A_{l,l}\neq 0$. Indeed, one can show using explicit formulas for Clebsch-Gordan coefficients that 
\begin{align*}
 \langle l,0, l, 0|2, 0\rangle &=(-1)^{l-1} (2l(l+1)) \sqrt{\frac{5(2l-2)!}{(2l+3)!}} , \\
 \langle l,l, l, -l|2, 0\rangle  &= (2l)! \sqrt{\frac{5}{(2l-2)!(2l+3)!}} 
\end{align*}
and therefore
\begin{align*}
d_{l,l,l,-l}^{2,0} = (-1)^{l-1}\sqrt{\frac{5}{4\pi}} \frac{l }{(2l+3)} \neq 0, \qquad l\geq 1.
\end{align*}
This finishes the proof of the proposition.
\end{proof}

Theorem \ref{thm:local_minimizer_ball3d} now follows from the previous two propositions.

\begin{proof}[Proof of Theorem \ref{thm:local_minimizer_ball3d}]
Case (1) of the theorem follows from Proposition \ref{prop:prop3d_Sn_exp_linear} which shows that the ball is not a local minimizer, as one can deform the ball to a domain $\Omega_\varepsilon$ with $S_n(\Omega_\varepsilon) < S_n(\mathbb{B})$ for small enough $\varepsilon>0$. 

In case (2), Proposition \ref{prop:prop3d_Sn_exp_quadratic} and the fact that $a_{0,0}=0$ show 
\begin{align*}
S_n(\Omega_\varepsilon) =  S_n(\mathbb{B}) + \varepsilon^2 \cdot \sum\limits_{L \geq 1,M} D_{L} |a_{L,M}|^2 + O(\varepsilon^3) \qquad (\varepsilon\to 0)
\end{align*}
with coefficients $D_{L}$ as stated in \eqref{eq:local_minimizer_ball_case2_DL}. We claim that $D_1 = 0$ for any $n$ since these modes correspond to infinitesimal translations. By definition,
\begin{align*} 
D_{1} &= |\mathbb{B}|^{2/3}\sum\limits_{j_{l+1/2,p}^2 \leq \lambda_n(\mathbb{B})} (2l+1) j_{l+1/2,p}^2 \Bigg(\frac{1}{ \pi}  + \frac{(2l+1)}{4\pi \cdot 3 }  \big(\langle l, 0, l,0|1,0\rangle \big)^2 \\
&\hspace{2.8cm}+ 2\sum\limits_{l'\neq l} j_{l+1/2,p}\frac{j_{l'}'(j_{l+1/2,p})}{j_{l'}(j_{l+1/2,p})} \frac{(2l'+1)}{4\pi \cdot 3 }  \big(\langle l, 0, l',0|1,0\rangle \big)^2  \Bigg)
\end{align*}
where $\langle l, 0, l,0|1,0\rangle = 0$, since $\langle l, 0, l,0|1,0\rangle = (-1)^{2l-1} \langle l, 0, l,0|1,0\rangle$ by Lemma \ref{lem:cg_coeff} (3). Selection rules imply that the sum $\sum_{l'\neq l}$ only contains $l' = l \pm 1$ if $l\geq 1$ or $l'=1$ if $l=0$. If $l=0$, then $ \langle 0, 0, 1,0|1,0\rangle = 1$ and from the recurrence relation $ j_1'(x) = j_{0}(x) - \frac{2}{x}j_1(x) $, we obtain 
\begin{align*}
    j_{1/2,p}\frac{j_{1}'(j_{1/2,p})}{j_{1}(j_{1/2,p})} = -2
\end{align*}
and therefore $D_1 = 0$. If $l \geq 1$, then 
\begin{align*}
    \big(\langle l, 0, l-1,0|1,0\rangle \big)^2 &= \frac{3l}{(2l-1)(2l+1)} , \\
    \big(\langle l, 0, l+1,0|1,0\rangle \big)^2 &= \frac{3(l+1)}{(2l+1)(2l+3)} 
\end{align*}
and from the recurrence relation $ j_\nu'(x) = j_{\nu-1}(x) - \frac{\nu+1}{x}j_\nu(x) = \frac{\nu}{x}j_\nu(x) - j_{\nu+1}(x)$ for $\nu \in \mathbb{R}$, we get
\begin{align*}
j_{l+1/2,p}\frac{j_{l-1}'(j_{l+1/2,p})}{j_{l-1}(j_{l+1/2,p})} &=   l-1, \\
    j_{l+1/2,p}\frac{j_{l+1}'(j_{l+1/2,p})}{j_{l+1}(j_{l+1/2,p})} &= - (l+2)  
\end{align*}
and thus $D_1 = 0$.

Hence, the ball is a weak local minimizer of $S_n(\Omega)$ if all second order perturbation coefficients $D_{L}$, $L\geq 2$, are positive. If $D_{L}$ is negative for some $L\geq 2$, then the ball is not a local minimizer but a saddle point, because perturbing the spherical boundary with $a_{L, M}\neq 0$ for some $-L \leq M \leq L$ decreases $S_n(\Omega)$ quadratically for small enough $\varepsilon$. 
\end{proof}

We close this section with an analogue of Lemma \ref{lem:dj_pos_2mmax}, showing that perturbations which introduce wild fluctuations to the surface of the ball give positive quadratic corrections to $S_n(\Omega)$.

\begin{lemma}\label{lem:Dj_pos_2mmax}
Let $n\in\mathbb{N}$ and $l_{max}(n) := \max\lbrace l : j_{l+1/2,1}^2 \leq \lambda_n(\mathbb{B})\rbrace $. Then
\begin{align*}
    D_{L}>0  
\end{align*}
for all $L \geq \big(1+\sqrt{4/3} \big) l_{max}(n)+6$.
\end{lemma}

\begin{proof}
The proof is similar to that of Lemma \ref{lem:dj_pos_2mmax}. It is sufficient to show that 
\begin{align*}
    \sum\limits_{l'\neq l} j_{l+1/2,p}\frac{j_{l'}'(j_{l+1/2,p})}{j_{l'}(j_{l+1/2,p})} \frac{(2l'+1)}{4\pi (2L+1)}  \left(\langle l, 0, l',0|L,0\rangle \right)^2 \geq 0
\end{align*}
for all $l\in \mathbb{N}_0$, $p\in\mathbb{N}$ such that $j_{l+1/2,p}^2 \leq \lambda_n(\mathbb{B})$, $L \geq \big(1+\sqrt{4/3} \big) l_{max}(n)+6$. Selection rules imply that in the sum above only terms with $|L-l|\leq l' \leq L+l$ survive. It is sufficient to show that $j_{l'}'(j_{l+1/2,p})/j_{l'}(j_{l+1/2,p}) > 0$ for all $|L-l|\leq l' \leq L+l$. The spherical Bessel equation 
$$ \frac{d}{dx} \bigg( x^2 \frac{d}{dx}j_\ell(x)\bigg) = (\ell ( \ell+1)-x^2 )j_\ell(x)$$
and the fact that $j_\ell(x)$ is positive for $ 0<x<\ell $ implies that $j_\ell'(x)$ and therefore $j_\ell'(x)/j_\ell(x)$ is positive for $ 0<x<\ell $. So let $l\in \mathbb{N}_0$, $p\in\mathbb{N}$ such that $j_{l+1/2,p}^2 \leq \lambda_n(\mathbb{B})$ and $L \geq \big(1+\sqrt{4/3} \big) l_{max}(n)+6$. Then, by \eqref{eq:besseljzero_m1_ineq}, it follows
\begin{align*}
    j_{l+1/2,p} \leq j_{l_{max}(n)+3/2,1} &< \sqrt{\frac{4}{3}\bigg(l_{max}(n)+\frac{5}{2}\bigg)\bigg(l_{max}(n)+\frac{13}{2}\bigg)} \\
    &<\sqrt{\frac{4}{3}} (l_{max}(n)+5) \\
    &< L - l_{max}(n) \\
    &\leq l' 
\end{align*}
and hence $j_{l'}'(j_{l+1/2,p})/j_{l'}(j_{l+1/2,p}) > 0$.
\end{proof}

\section{Infinite families of saddle points} \label{sec:infinite_families}

Theorem \ref{thm:saddle_point_infinitely_many} is a consequence of the following result. We prove that in two dimensions elliptical ($j=2$) perturbations generate an infinite number of cases where the associated second order perturbation coefficient is negative.

\begin{proposition}\label{prop:saddle_before_simple_n}
Let $n \in \mathbb{N}$ and let $d_j$ be given by \eqref{eq:local_minimizer_disc_case2_dj}. If $\lambda_{n+1}(\mathbb{D})$ is a single eigenvalue, i.e.~$\lambda_n(\mathbb{D})<\lambda_{n+1}(\mathbb{D})$ and $\lambda_{n+1}(\mathbb{D}) = j_{0p}^2$ for some $p\geq 2$, then $ d_2+ d_{-2} <0$. 
\end{proposition} 

By Theorem \ref{thm:local_minimizer_disc}, this means that the disk is a saddle point of $S_n(\Omega)$ if $\lambda_{n+1}(\mathbb{D})$ is a single eigenvalue - and the disk has infinitely many of them.

To prove the proposition, we first write $d_2+ d_{-2}$ in closed form, exploiting the specific case $j=2$.

\begin{lemma} \label{lem:d2_explicit_sumterms}
Let $n \in \mathbb{N}$ such that $\lambda_n(\mathbb{D})<\lambda_{n+1}(\mathbb{D})$. Then
    \begin{align*}
        d_2 + d_{-2} =  \sum_{j_{0p}^2 \leq \lambda_n(\mathbb{D})} 2\pi j_{0p}^2 (j_{0p}^2-2) +  \sum_{j_{1p}^2 \leq \lambda_n(\mathbb{D})} \pi j_{1p}^4 -    \sum\limits_{\substack{j_{mp}^2 \leq \lambda_n(\mathbb{D}), \\ m \geq 2}}  4 \pi j_{mp}^2 \bigg(2  + \frac{j_{mp}^2}{m^2-1} \bigg).
    \end{align*}
\end{lemma}

\begin{proof}
    Recall the recurrence relations $J_{\nu-1}(x)+ J_{\nu+1}(x) = \frac{2\nu}{x} J_\nu(x)$, $J_\nu'(x) = J_{\nu-1}(x) - \frac{\nu}{x}J_\nu(x) = \frac{\nu}{x}J_\nu(x) - J_{\nu+1}(x)$ for $\nu \in \mathbb{R}$. At $x = j_{mp}$, these yield $J_{m}'(j_{mp}) = J_{m-1}(j_{mp}) = - J_{m+1}(j_{mp}) \neq 0$ and 
    \begin{align*}
        j_{mp} \frac{J_{m+2}'(j_{mp})}{J_{m+2}(j_{mp})} &= \frac{j_{mp}^2}{2(m+1)} - (m+2) \qquad (m \geq 0),\\
        j_{mp} \frac{J_{m-2}'(j_{mp})}{J_{m-2}(j_{mp})} &= (m-2) - \frac{j_{mp}^2}{2(m-1)}\qquad (m \geq 2).
    \end{align*}
    The statement of the lemma follows after inserting these expressions accordingly into the identity
    \begin{align*}
          d_2 + d_{-2} &=\sum\limits_{j_{0p}^2 \leq \lambda_n(\mathbb{D})} 2\pi j_{0p}^2 \left(2 +  2 j_{0p}\frac{J_2'(j_{0p})}{J_2(j_{0p})}   \right)  \\
& \qquad+ \sum\limits_{j_{1p}^2 \leq \lambda_n(\mathbb{D})} 4\pi j_{1p}^2 \left( 3  +  j_{1p}\frac{J_{3}'(j_{1p})}{J_{3}(j_{1p})}  \right) \\
& \qquad+ \sum\limits_{\substack{j_{mp}^2 \leq \lambda_n(\mathbb{D}), \\ m \geq 2}} 4\pi j_{mp}^2 \left( 2 +   j_{mp}\frac{J_{m-2}'(j_{mp})}{J_{m-2}(j_{mp})}  +j_{mp}\frac{J_{m+2}'(j_{mp})}{J_{m+2}(j_{mp})}  \right) .
    \end{align*}
\end{proof}

The second ingredient is an a priori bound on $d_2+ d_{-2}$.

\begin{lemma} \label{lem:apriori_d2_bound}
Let $n \in \mathbb{N}$ such that $\lambda_n(\mathbb{D}) < \lambda_{n+1}(\mathbb{D})$. Then
    \begin{align*}
        d_2 + d_{-2} \leq 8 S_n(\mathbb{D}).
    \end{align*}
\end{lemma}

\begin{proof}
    Consider the ellipse $E_\varepsilon = \{ (e^\varepsilon x, e^{-\varepsilon} y) \, : \, (x,y) \in\mathbb{D} \}$ for small $\varepsilon$. Note that $|E_\varepsilon| = \pi $. Its Dirichlet eigenvalues $\lambda_k(E_\varepsilon)$ coincide with those of the quadratic form
    \begin{align*}
        q_\varepsilon[u]=\int_{\mathbb{D}} e^{-2\varepsilon} |\partial_x u|^2 +  e^{2\varepsilon} |\partial_y u|^2 \, d(x,y), \qquad u \in H_0^1(\mathbb{D}). 
    \end{align*}
    Let $u_1, ... , u_n$ be an orthonormal set of Dirichlet eigenfunctions of $\mathbb{D}$. By the variational principle for sums of eigenvalues, see \cite[Proposition 1.33]{Frank2022}, we obtain
    \begin{align*}
         \sum_{k=1}^n \lambda_k(E_\varepsilon) \leq  \sum_{k=1}^n q_\varepsilon[u_k] = e^{-2\varepsilon}\sum_{k=1}^n \Vert \partial_x u_k \Vert_{L^2(\mathbb{D})}^2  + e^{2\varepsilon}\sum_{k=1}^n \Vert \partial_y u_k\Vert_{L^2(\mathbb{D})}^2  .
    \end{align*}
    Since $\lambda_n(\mathbb{D}) < \lambda_{n+1}(\mathbb{D})$, the space $V:= \operatorname{span}\{ u_1, ..., u_n\}$ is invariant under rotations of $\pi/2$, thus 
    \begin{align*}
        \sum_{k=1}^n\Vert \partial_x u_k\Vert_{L^2(\mathbb{D})}^2 = \sum_{k=1}^n\Vert \partial_y u_k\Vert_{L^2(\mathbb{D})}^2 = \frac{1}{2} \sum_{k=1}^n \Vert \nabla u_k\Vert_{L^2(\mathbb{D})}^2 
    \end{align*}
    and therefore
     \begin{align} \label{eq:S_n_E_eps_1}
        S_n(E_\varepsilon) \leq S_n(\mathbb{D}) \cosh(2\varepsilon) = S_n(\mathbb{D}) (1+ 2 \varepsilon^2 + O(\varepsilon^4)) \qquad (\varepsilon \to 0).
    \end{align}
    On the other hand, the ellipse $E_\varepsilon$ can be described by the radius function
    \begin{align*}
        r(\varphi,\varepsilon) &= (\cosh(2\varepsilon) - \sinh(2\varepsilon) \cos(2\varphi))^{-1/2} \\
        &= 1 + \varepsilon \cos(2\varphi ) + \varepsilon^2 \bigg( - \frac{1}{4} + \frac{3}{4} \cos(4\varphi)\bigg) + O(\varepsilon^3) \qquad (\varepsilon \to 0),
    \end{align*}
    where the remainder is uniform in $\varphi$. Let us consider $\Omega_\varepsilon$ obtained by the perturbation coefficients $a_{\pm 2} = \frac{1}{2}$, $b_0 = - \frac{1}{4}$, $b_{\pm 4} = \frac{3}{8}$ and all other $a_j$, $b_j$ set to zero. Then by Proposition \ref{prop:prop_Sn_exp} (2), 
    \begin{align*} 
        S_n(\Omega_\varepsilon) = S_n(\mathbb{D}) + \frac{\varepsilon^2}{4} (d_2 + d_{-2}) +O(\varepsilon^3) \qquad (\varepsilon \to 0).
    \end{align*}
    Since the radius functions of $E_\varepsilon$ and $\Omega_\varepsilon$ differ only by a factor of $1+O(\varepsilon^3)$ uniformly in $\varphi$, we conclude $(1-C\varepsilon^3)\Omega_\varepsilon \subset E_\varepsilon \subset (1+C\varepsilon^3) \Omega_\varepsilon $ for some $C>0$ and by domain monotonicity and scaling of eigenvalues, 
    \begin{align} \label{eq:S_n_E_eps_2}
        S_n(E_\varepsilon) =S_n(\Omega_\varepsilon)  (1+ O(\varepsilon^3)) =  S_n(\mathbb{D}) + \frac{\varepsilon^2}{4} (d_2 + d_{-2}) +O(\varepsilon^3) \qquad (\varepsilon \to 0).
    \end{align}
    Comparing \eqref{eq:S_n_E_eps_1} and \eqref{eq:S_n_E_eps_2}, we obtain $ d_2 + d_{-2} \leq 8 S_n(\mathbb{D})$.
\end{proof}

\begin{proof}[Proof of Proposition \ref{prop:saddle_before_simple_n}]
    Let $d_2(n)= d_2 + d_{-2}$ evaluated for $S_n(\Omega)$. Then $ d_2(n+1) = d_2(n) + 2 \pi j_{0p}^2 (j_{0p}^2-2)$ by Lemma \ref{lem:d2_explicit_sumterms} and therefore by Lemma \ref{lem:apriori_d2_bound}, 
    \begin{align*}
        d_2(n) \leq 8 S_{n+1}(\mathbb{D}) -  2 \pi j_{0p}^2 (j_{0p}^2 - 2 ).
    \end{align*}
    The validity of P{\'o}lya's conjecture on the disk $\mathbb{D}$, see \cite{Filonov2022}, shows that $n+1 \leq j_{0p}^2/4$, so
    \begin{align*}
        S_{n+1}(\mathbb{D}) \leq \pi \lambda_1 (\mathbb{D}) + n \pi \lambda_{n+1}(\mathbb{D}) \leq \pi j_{01}^2 + \frac{\pi}{4} j_{0p}^4 - \pi j_{0p}^2
    \end{align*}
    and hence
    \begin{align*}
        d_2(n) \leq 4\pi (2 j_{01}^2 -  j_{0p}^2) < 0
    \end{align*}
    for $p\geq 2$, since $2 j_{01}^2 \approx 11.57 $ and $ j_{02}^2 \approx 30.47$.
\end{proof}

The same strategy can be used to find a second infinite family of $n$'s for which the disk is a saddle point of $S_n(\Omega)$. If $\lambda_n(\mathbb{D})<\lambda_{n+1}(\mathbb{D})$ and $\lambda_{n+1}(\mathbb{D}) = \lambda_{n+2}(\mathbb{D}) = j_{1p}^2$, then $ d_2(n) \leq 8 S_{n+2}(\mathbb{D}) -   \pi j_{1p}^4 $. Two-term asymptotics for sums of Dirichlet eigenvalues and the validity of P{\'o}lya's conjecture on the disk imply that $ d_2(n)<0 $ for sufficiently large $n$.

Surprisingly, in three dimensions, a result parallel to Proposition \ref{prop:saddle_before_simple_n} is false. Whereas in two dimensions, elliptical perturbations generate an infinite family of saddle points, ellipsoidal perturbations are asymptotically positive definite in three dimensions and do not generate an infinite family of saddle points in the same way.

\begin{proposition} \label{prop:D2toinfinity}
 Let $n\in\mathbb{N}$ such that $\lambda_n(\mathbb{B})< \lambda_{n+1}(\mathbb{B})$. Then
\begin{align*}
    D_2 \geq \frac{9\lambda_n(\mathbb{B})^{5/2} }{250\pi^2}(1+o(1))  \qquad (n\to \infty)
\end{align*}
and in particular $D_2 \to \infty$ as $n \to \infty$.
\end{proposition}

To prove Proposition \ref{prop:D2toinfinity}, we use an equivalent representation of $D_2$.

\begin{lemma}
    Let $n\in\mathbb{N}$ such that $\lambda_n(\mathbb{B})< \lambda_{n+1}(\mathbb{B})$. Then
\begin{align*}
    D_2 = \sum\limits_{j_{l+1/2,p}^2 \leq \lambda_n(\mathbb{B})} f_l(j_{l+1/2,p})
\end{align*}
where
\begin{align*}
   f_l (x) =  \frac{(2l+1) x^2}{2\pi (2l-1)(2l+3)} \bigg( \frac{3x^2}{ (2l-1)(2l+3)} - (2l^2+2l-3)\bigg) .
\end{align*}
\end{lemma}

\begin{proof}
    By selection rules and Corollary \ref{cor:d_coeff_lsum_odd}, the sum $\sum_{l'\neq l}$ in \eqref{eq:local_minimizer_ball_case2_DL} for $L=2$ only contains terms with $l' = l\pm 2$ if $l\geq 2$, $l'=3$ if $l=1$ or $l'=2$ if $l=0$. Furthermore,
    \begin{align*}
        (\langle l, 0, l,0|2,0\rangle )^2 &= \frac{5l(l+1)}{(2l-1)(2l+1)(2l+3)}, \qquad l\geq 0, \\
        (\langle l, 0, l-2,0|2,0\rangle )^2 &= \frac{15l(l-1)}{2(2l-3)(2l-1)(2l+1)},  \qquad l\geq 2,\\
        (\langle l, 0, l+2,0|2,0\rangle )^2 &= \frac{15(l+1)(l+2)}{2(2l+1)(2l+3)(2l+5)},  \qquad l\geq 0,
    \end{align*}
    and from recurrence relations of spherical Bessel functions follows
    \begin{align*}
        j_{l+1/2,p} \frac{j_{l+2}'(j_{l+1/2,p})}{j_{l+2}(j_{l+1/2,p})} &= \frac{j_{l+1/2,p}^2}{2l+3} - (l+3) , \qquad \qquad l \geq 0, \\
        j_{l+1/2,p} \frac{j_{l-2}'(j_{l+1/2,p})}{j_{l-2}(j_{l+1/2,p})} &= (l-2) - \frac{j_{l+1/2,p}^2}{2l-1}, \qquad\qquad l\geq 2.
    \end{align*}
    Inserting these expressions in \eqref{eq:local_minimizer_ball_case2_DL} for $L=2$ yields the desired closed form of $D_2$.
\end{proof}

\begin{proof}[Proof of Proposition \ref{prop:D2toinfinity}]
With the closed form sum for $D_2$ obtained in the previous lemma, we split its contributions into
\begin{align*}
     D_2 = \sum\limits_{j_{1/2,p}^2 \leq \lambda_n(\mathbb{B})} f_0(j_{1/2,p}) +  \sum\limits_{j_{3/2,p}^2 \leq \lambda_n(\mathbb{B})} f_1(j_{3/2,p}) +  \sum\limits_{\substack{j_{l+1/2,p}^2 \leq \lambda_n(\mathbb{B}), \\l\geq 2}} f_l(j_{l+1/2,p}).
\end{align*}
Omitting the positive $x^4$-terms of $f_l(x)$ for $l\geq 2$ and estimating $(2l^2+2l-3)/[(2l-1)(2l+3)]<1/2$, we get
$f_l(x) > -(2l+1) x^2/(4\pi) $, $l\geq 2$, and therefore 
\begin{align*}
    \sum\limits_{\substack{j_{l+1/2,p}^2 \leq \lambda_n(\mathbb{B}), \\l\geq 2}} f_l(j_{l+1/2,p}) > -\frac{1 }{4\pi} \sum\limits_{\substack{j_{l+1/2,p}^2 \leq \lambda_n(\mathbb{B}), \\l\geq 2}} (2l+1)j_{l+1/2,p}^2 \geq -  \frac{1 }{4\pi} \sum_{k=1}^n \lambda_k(\mathbb{B})
\end{align*}
where by Weyl's law,
\begin{align*}
    \sum_{k=1}^n \lambda_k(\mathbb{B}) = \frac{2}{15\pi} \lambda_{n}(\mathbb{B})^{5/2} (1+o(1)) \qquad (n\to \infty)
\end{align*}
and thus
\begin{align} \label{eq:sum_fl_lower_bnd}
    \sum\limits_{\substack{j_{l+1/2,p}^2 \leq \lambda_n(\mathbb{B}), \\l\geq 2}} f_l(j_{l+1/2,p}) \geq  -\frac{\lambda_{n}(\mathbb{B})^{5/2}}{30 \pi^2 }  (1+o(1))  \qquad (n\to \infty).
\end{align}
For $l=0$ and $l=1$, the functions $f_l(x)$ are
\begin{align*}
    f_0(x) = \frac{x^4}{6\pi} - \frac{x^2}{2\pi},  \qquad \qquad    f_1(x) = \frac{9x^4}{50\pi} - \frac{3x^2}{10\pi},
\end{align*}
and since $j_{1/2,p} = p \pi$ and $j_{3/2,p} \in (p\pi, (p+\frac{1}{2})\pi)$ for $p \in \mathbb{N}$, we have
\begin{align*} 
    \sum\limits_{j_{1/2,p}^2 \leq \lambda_n(\mathbb{B})} f_0(j_{1/2,p}) &= \frac{1}{6\pi} \bigg(\sum\limits_{j_{1/2,p}^2 \leq \lambda_n(\mathbb{B})} j_{1/2,p}^4 \bigg) (1+o(1)) = \frac{\lambda_n(\mathbb{B})^{5/2}}{30\pi^2} (1+o(1)), \\
    \sum\limits_{j_{3/2,p}^2 \leq \lambda_n(\mathbb{B})} f_1(j_{3/2,p}) &= \frac{9}{50\pi} \bigg(\sum\limits_{j_{3/2,p}^2 \leq \lambda_n(\mathbb{B})} j_{3/2,p}^4\bigg) (1+o(1)) = \frac{9\lambda_n(\mathbb{B})^{5/2} }{250\pi^2}(1+o(1)) 
\end{align*}
as $n\to \infty$. Combining these two asymptotics with the asymptotic lower bound for $l\geq2$ from  \eqref{eq:sum_fl_lower_bnd} yields the claim of the proposition.
\end{proof}

We remark that with a bit more effort, one could even show that $D_2>0$ for any $n\in\mathbb{N}$ (which is supported by numerics). First, one would prove that there exist a small explicit constant $c>0$ and a moderately large $n_0\in\mathbb{N}$ such that
\begin{align*}
    D_2 \geq c \lambda_n(\mathbb{B})^{5/2}, \qquad n \geq n_0.
\end{align*}
For this, one needs to replace the asymptotics used in the proof above by quantitative estimates which in the case of the $l\geq 2$ sum can be obtained by the validity of P{\'o}lya's conjecture for the ball whereas for $f_0(x)$ and $f_1(x)$ elementary lower bounds suffice. After that one would need to check the cases $n<n_0$ by certified numerical evaluations in interval arithmetic.

\appendix

\section{Numerical addendum} \label{app:numerical_addendum}

In the tables below we report numerical sample counts for how often the disk and the ball fall into one of the categories of being a non-critical shape, a critical shape, a weak local minimizer, and a saddle point for $S_n(\Omega)$. In both $d=2$ and $d=3$, the number of saddle point cases appears to outweigh the number of weak local minimizer cases in the long run.

\begin{table}[ht]
    \begin{center}
\begin{tabular}{l|rrrr}
$n\leq $   & non-critical & critical & weak loc.~min. & saddle point \\ \hline
$10^1$     & $4$  & $6$ & $3 $ & $ 3$ \\
$10^2$     & $47$ & $53$ & $ 17$ & $36 $ \\
$10^3$     & $490$ & $510$ & $ 87$ & $423 $\\
$10^4$     & $4{,}968$ & $5{,}032$& $ 604$ & $ 4{,}428$\\
$10^5$     & $49{,}900$ & $50{,}100$ & $3{,}637 $ & $ 46{,}463$\\
$10^6$     & $499{,}682$ & $500{,}318$ & $ 14{,}050 $ & $ 486{,}268$\\
\end{tabular}
\end{center}
    \caption{Counts for the disk $\mathbb{D}$. }
    \label{tab:disk_ntable}
\end{table}

\begin{table}[ht]
\begin{center}
\begin{tabular}{l|rrrr}
$n\leq $   & non-critical & critical & weak loc.~min. & saddle point \\ \hline
$10^1$     & $6$          & $4$        & $4$        & $0$ \\
$10^2$     & $83$         & $17$       & $12$       & $5$ \\
$10^3$     & $927$        & $73$       & $43$       & $30$ \\
$10^4$     & $9{,}656$    & $344$      & $165$      & $179$ \\
$10^5$     & $98{,}420$   & $1{,}580$  & $665$      & $915$ \\
$10^6$     & $992{,}684$  & $7{,}316$  & $2{,}804$  & $4{,}512$ \\
\end{tabular}
\end{center}
    \caption{Counts for the ball $\mathbb{B}$. }
    \label{tab:ball_ntable}
\end{table}

\section*{AI usage disclosure}

The author acknowledges the use of Claude Fable 5.1 for mathematical discussions and editorial assistance. In Sections 1-3, Claude has been used only for auditing and reviewing text by the author. In Section 4, Claude majorly contributed to the results and proofs of the section. All AI-assisted arguments and computations were independently checked by the author, who takes full responsibility for the contents of the paper.

\newpage

\printbibliography

\end{document}